\documentclass[11pt,letterpaper]{amsart}

\usepackage[T1]{fontenc}
\usepackage{lmodern}
\usepackage{hyperref}
\usepackage{etex}
\usepackage[shortlabels]{enumitem}
\usepackage{amsmath}
\usepackage{amsxtra}
\usepackage{amscd}
\usepackage{amsthm}
\usepackage{amsfonts}
\usepackage{amssymb}
\usepackage{eucal}
\usepackage[all]{xy}
\usepackage{graphicx}
\usepackage{tikz-cd}
\usepackage{mathrsfs}
\usepackage{subfiles}
\usepackage{mathpazo}
\usepackage[colorinlistoftodos, textsize=tiny]{todonotes}
\usepackage{morefloats}
\usepackage{pdfpages}
\usepackage{thm-restate}
\usepackage[utf8]{inputenc}
\usepackage{epigraph}
\usepackage{csquotes}
\usepackage[margin=1in]{geometry}
\usepackage{adjustbox}
\usepackage{scalerel}
\usepackage{stackengine}
\stackMath
\newcommand\reallywidehat[1]{%
\savestack{\tmpbox}{\stretchto{%
  \scaleto{%
    \scalerel*[\widthof{\ensuremath{#1}}]{\kern-.6pt\bigwedge\kern-.6pt}%
    {\rule[-\textheight/2]{1ex}{\textheight}}
  }{\textheight}%
}{0.5ex}}%
\stackon[1pt]{#1}{\tmpbox}%
}
\graphicspath{ {images/} }

\RequirePackage{color}
\definecolor{myred}{rgb}{0.75,0,0}
\definecolor{mygreen}{rgb}{0,0.5,0}
\definecolor{myblue}{rgb}{0,0,0.65}

\usepackage{color}

\usepackage{hyperref}
\hypersetup{citecolor=blue}
\usepackage{tikz}
\usetikzlibrary{matrix,arrows,decorations.pathmorphing}

\theoremstyle{plain}
\newtheorem{theorem}[subsection]{Theorem}
\newtheorem{slogan}[subsection]{Slogan}
\newtheorem{proposition}[subsection]{Proposition}
\newtheorem{lemma}[subsection]{Lemma}
\newtheorem{corollary}[subsection]{Corollary}

\theoremstyle{definition}
\newtheorem{definition}[subsection]{Definition}
\newtheorem{remark}[subsection]{Remark}
\newtheorem{construction}[subsection]{Construction}
\newtheorem{example}[subsection]{Example}

\newtheorem{question}[subsection]{Question}
\newtheorem{conjecture}[subsection]{Conjecture}

\theoremstyle{remark}
\newtheorem{notation}[subsection]{Notation}
\numberwithin{equation}{section}
\newcommand\nc{\newcommand}
\nc\on{\operatorname}
\nc\renc{\renewcommand}

\newcommand*{\shom}{\mathscr{H}\kern -.5pt om}
\newcommand*{\stor}{\mathscr{T}\kern -.5pt or}
\newcommand*{\sext}{\mathscr{E}\kern -.5pt xt}

\makeatletter
\providecommand\@dotsep{5}
\renewcommand{\listoftodos}[1][\@todonotes@todolistname]{%
\@starttoc{tdo}{#1}}
\makeatother

\makeatletter
\newcommand{\customlabel}[2]{\protected@write \@auxout {}{\string \newlabel {#1}{{#2}{\thepage}{#2}{#1}{}} }\hypertarget{#1}{#2}}

\DeclareMathOperator\id{id}

\DeclareMathOperator\rk{rk}

\renewcommand\sp{\mathrm{Sp}}

\DeclareMathOperator\gl{GL}

\renewcommand\sl{\mathrm{SL}}

\DeclareMathOperator\so{SO}

\renewcommand\o{{\rm{O}}}

\newcommand\nbound{\frac{3r^2}{\sqrt{g+1}}+8r}

\DeclareFontFamily{U}{wncy}{}
\DeclareFontShape{U}{wncy}{m}{n}{<->wncyr10}{}
\DeclareSymbolFont{mcy}{U}{wncy}{m}{n}
\DeclareMathSymbol{\Sha}{\mathord}{mcy}{"58}

\def\listtodoname{List of Todos}
\def\listoftodos{\@starttoc{tdo}\listtodoname}

\title{Big monodromy for higher Prym representations}

\author{Aaron Landesman, Daniel Litt, and Will Sawin}

\usepackage{microtype}
\begin{document}

\begin{abstract} 
Let $\Sigma_{g'}\to \Sigma_g$ be a cover of an orientable surface of genus $g$
by an orientable surface of genus $g'$, branched at $n$ points, with Galois group $H$. 
Such a cover induces a virtual action of the mapping class group $\on{Mod}_{g,n+1}$ of a genus $g$ surface with $n+1$ marked points on $H^1(\Sigma_{g'}, \mathbb{C})$. 
When $g$ is large in terms of the group $H$,
we calculate precisely the connected monodromy group of this action.
The methods are Hodge-theoretic 
and rely on a ``generic Torelli theorem with coefficients.''
\end{abstract}

\maketitle

\setcounter{tocdepth}{1}
\tableofcontents

\section{Introduction}

A classical result in geometric topology
states that the
action of the mapping class group $\on{Mod}_g$ of a surface $\Sigma_g$ of genus $g$ on its
first cohomology, $H^1(\Sigma_g, \mathbb{Z})$, is via the full group of
automorphisms preserving the cup product, namely $\on{Sp}_{2g}(\mathbb{Z})$. 
Let the hyperelliptic mapping class group denote the subgroup of mapping classes 
commuting with an involution having genus zero quotient.
If
one restricts to the hyperelliptic mapping class group the image is still of finite index in $\on{Sp}_{2g}(\mathbb{Z})$
\cite{a1979tresses}. Finally, one may consider the monodromy representation on
the cohomology of Prym varieties arising from connected \'etale double covers of genus $g$
curves. It follows from \cite{looijenga:prym-representations} that the image of this representation is finite index in $\on{Sp}_{2g-2}(\mathbb{Z})$. 


In order to generalize the above cases, fix an \emph{arbitrary} finite group $H$ and consider
a maximal family of (possibly branched) Galois $H$-covers of curves of genus $g$. What is the monodromy representation on the 
first cohomology of these covering curves? The three cases considered above
correspond to the very special cases $H=\{\id\}$, the case $g=0$ and
$H=\mathbb{Z}/2\mathbb{Z}$, and the case where $g\geq 2$, $H=\mathbb{Z}/2\mathbb{Z}$, and the covers are unramified. 
Our main result asserts that once the genus $g$ of the base curve is sufficiently
large, the connected monodromy group of this family of $H$-covers is as large as
possible.
Namely, just as the action of the mapping class group on $H^1(\Sigma_g, \mathbb{Z})$ cannot be via all of
$\gl_{2g}(\mathbb Z)$ because it must preserve the cup product, a symplectic form, the monodromy
of families of $H$-covers cannot be the full general linear group, but must
preserve the symplectic form and respect the $H$-action. Moreover, as local
systems of geometric origin are semisimple, it must be
semisimple. 
These considerations show that the identity component of the Zariski
closure of the image of the monodromy representation is
contained in the derived subgroup of the centralizer of $H$ in the symplectic
group.
We will show that it is in fact equal to this group, once the genus of the base
curve is
sufficiently large.

\subsection{Statement of results}
We next set up notation to state our main results more precisely.
Let $\Sigma_{g,n}$ be an orientable surface of genus $g$, with $n$ punctures,
and let $\on{Mod}_{g,n}$ be the pure mapping class group of $\Sigma_{g,n}$. That is, $\on{Mod}_{g,n}=\pi_0(\text{Homeo}^+(\Sigma_{g,n}))$, where $\text{Homeo}^+(\Sigma_{g,n})$
is the space of orientation preserving homeomorphisms of $\Sigma_{g}=\Sigma_{g,0}$ fixing
each of the $n$ punctures. 
The goal of this paper is to study certain natural
representations of these mapping class groups, arising from finite unramified covers of
$\Sigma_{g,n}$.
We call these representations \emph{higher Prym representations}, following the terminology of \cite{putmanW:abelian-quotients}.

Let $H$ be a finite group and let $\Sigma_{g',n'}\to \Sigma_{g,n}$ be a
finite unramified Galois
$H$-cover. 
We will next construct a homomorphism from a finite index subgroup of
$\on{Mod}_{g,n+1}$ to the centralizer of $H$ in
$\on{Sp}(H^1(\Sigma_{g'}, \mathbb{C}))$.
Fix a point $x$ of $\Sigma_{g,n}$. All finite unramified $H$-torsors
$\Sigma_{g',n'}\to \Sigma_{g,n}$ arise from surjection $\varphi:
\pi_1(\Sigma_{g,n}, x)\twoheadrightarrow H$, with isomorphism of torsors corresponding to conjugation (by $H$) of surjections. 
For any $x' \in \Sigma_{g',n'}$ mapping to $x$, we can identify the kernel of
$\varphi$ with $K:=\pi_1(\Sigma_{g',
n'},x')$. 
The mapping class group $\on{Mod}_{g,n+1}$ acts (up to isotopy) on
$(\Sigma_{g,n}, x)$, 
where we view $x$ as an $(n+1)$-st marked point of $\Sigma_{g}$.
Thus,
$\on{Mod}_{g,n+1}$ acts
on $\pi_1(\Sigma_{g,n}, x)$, and hence on the
finite set of homomorphisms $\pi_1(\Sigma_{g,n}, x)\twoheadrightarrow H$. The
stabilizer $\on{Mod}_\varphi \subset \on{Mod}_{g,n+1}$ of $\varphi$ acts on the
kernel $K$ of $\varphi$.  The induced action on $K^{\text{ab}}=H_1(\Sigma_{g',
n'}, \mathbb{Z})$ preserves the kernel of the natural morphism $$H_1(\Sigma_{g',
n'}, \mathbb{Z})\to H_1(\Sigma_{g'}, \mathbb{Z}),$$ and thus $\varphi$ gives
rise to a virtual action of $\on{Mod}_{g,n+1}$ on $H_1(\Sigma_{g'},
\mathbb{Z})$, i.e.~an action of $\on{Mod}_{\varphi}$ on $H_1(\Sigma_{g'}, \mathbb{Z})$. 
 This action manifestly commutes with the action of $H$ and thus defines  a homomorphism
 \begin{align}
	\label{equation:monodromy-map}
R_\varphi: \on{Mod}_{\varphi}\to \on{Sp}(H^1(\Sigma_{g'}, \mathbb{C}))^H,
\end{align}
where $\on{Mod}_{\varphi}$ is the stabilizer of $\varphi$ as defined above, and $\on{Sp}(H^1(\Sigma_{g'}, \mathbb{C}))^H$ denotes the centralizer of the action of $H$ in $\on{Sp}(H^1(\Sigma_{g'}, \mathbb{C}))$. We may equivalently think of this as a virtual homomorphism from $\on{Mod}_{g,n+1}$ to $\on{Sp}(H^1(\Sigma_{g'}, \mathbb{C}))^H$, where a virtual homomorphism is a homomorphism from a finite-index subgroup. 

We are interested in the \emph{connected monodromy group} of this virtual action
of $\on{Mod}_{g,n+1}$, or, in other words, the identity component of the Zariski closure of the image of $R_\varphi$ inside $\on{Sp}(H^1(\Sigma_{g'}, \mathbb{C}))^H$. The slogan for this work is:
\begin{slogan}\label{slogan:big-monodromy}
\emph{Monodromy groups should be as big as possible.}
\end{slogan}
Based on this, we expect that, outside of a few exceptional cases, this Zariski closure should contain the commutator subgroup of $\on{Sp}(H^1(\Sigma_{g'}, \mathbb{C}))^H$. (After passing to a further finite-index subgroup, we may assume the monodromy group is connected, and connected components of monodromy groups of local systems of geometric origin are semisimple, so the commutator subgroup is the largest possible.) Our main result is that this expectation holds under a lower bound on the genus:
\begin{theorem}\label{theorem:main-thm-1}
Let $H$ be a finite group and let $\Sigma_{g',n'}\to \Sigma_{g,n}$ be an
$H$-cover associated to a surjection $\varphi: \pi_1(\Sigma_{g,n},
x)\twoheadrightarrow H$, where $x$ is a base point of $\Sigma_{g,n}$. Let $\overline{r}$ be the maximal dimension of an irreducible representation of $H$. Suppose that either \begin{enumerate}
 \item $n=0$ and $g\geq
2 \overline{r}+2$, or
\item $n$ is arbitrary and $g> \on{max}(2\overline{r}+1, \overline{r}^2).$
 \end{enumerate}
   Let $\on{Mod}_\varphi$ be the stabilizer of $\varphi$ inside $\on{Mod}_{g,n+1}$.
 Then,
 using notation as in \eqref{equation:monodromy-map},
 the identity component of the Zariski closure of the image of $$R_\varphi: \on{Mod}_{\varphi}\to \on{Sp}(H^1(\Sigma_{g'}, \mathbb{C}))^H$$ is the commutator subgroup of $\on{Sp}(H^1(\Sigma_{g'}, \mathbb{C}))^H$.
\end{theorem}

We prove \autoref{theorem:main-thm-1} in \autoref{subsection:main-proofs}.
We refer to representations as in \autoref{theorem:main-thm-1} as {\em
higher Prym representations}.

\begin{remark}
	\label{remark:}
	Note that the constant $\overline{r}$ reflects the group-theoretic properties of $H$. 
For example, $\overline{r}\leq \sqrt{\# H}$, and $\overline{r}$ divides the index of any normal abelian subgroup of $H$.
	The case that the covering group $H$ is abelian was considered in
\cite{looijenga:prym-representations}. In this case, the representation was
called a Prym representation.
Note that when $H$ is abelian, $\overline{r} = 1$, and so it suffices to take $g
\geq 4$ in the statement of \autoref{theorem:main-thm-1}.
If $H$ is dihedral, then $\overline{r}=2$.
\end{remark}

\begin{remark}\label{rem:virtual-image}
	One may reformulate \autoref{theorem:main-thm-1} as follows: the Zariski-closure of the virtual image of $\on{Mod}_{g, n+1}$ under $R_\varphi$ in $\on{Sp}(H^1(\Sigma_{g'}, \mathbb{C}))^H$ is the commutator subgroup of $\on{Sp}(H^1(\Sigma_{g'}, \mathbb{C}))^H$. Here the \emph{Zariski-closure of the virtual image} is the intersection of the Zariski-closures of the images of all finite-index subgroups on which $R_\varphi$ is defined.
\end{remark}

\begin{remark}
	\label{remark:putman-wieland}
	The Putman-Wieland conjecture \cite{putmanW:abelian-quotients} is
	heavily influenced by \autoref{slogan:big-monodromy}. It predicts that if
	$g>2$, the virtual action of $\on{Mod}_{g,n+1}$ on $H^1(\Sigma_{g'},
	\mathbb C)$ has no nonzero finite orbits. (In
		\cite{putmanW:abelian-quotients}, this conjecture was made for
		all $g
		\geq 2$. However, Markovi\'{c} \cite[Theorem 1.3]{markovic:unramified-correspondences}
		gave a counterexample when $g = 2$.)
		Note that the virtual action of $\on{Mod}_{g,n+1}$ on $H^1(\Sigma_{g'},
	\mathbb C)$ has no nonzero finite orbits if and only if the Zariski
	closure of $\on{Mod}_{g,n+1}$ in $\on{Aut}(H^1(\Sigma_{g'},
	\mathbb C))$ has no nonzero finite orbits.
	Also note that the commutator subgroup of $\on{Sp}(H^1(\Sigma_{g'}, \mathbb{C}))^H$ has
	no nonzero finite orbits on 
	$H^1(\Sigma_{g'}, \mathbb C)$ using
	\eqref{equation:symplectic-centralizer}.
	Therefore, the Putman-Wieland conjecture for a given covering
	$\Sigma_{g',n'} \to \Sigma_{g,n}$ is implied
	by the statement that the Zariski closure of 
	$\on{Mod}_{g,n+1}$ in $\on{Aut}(H^1(\Sigma_{g'},
	\mathbb C))$ is the commutator subgroup of $\on{Sp}(H^1(\Sigma_{g'}, \mathbb{C}))^H$.

In particular, since the analogue of the Putman-Wieland conjecture does not hold
for $g=2$ due to 
Markovi\'{c}'s counterexample
\cite[Theorem 1.3]{markovic:unramified-correspondences}, as mentioned above,
it follows that no analogue of \autoref{theorem:main-thm-1} can hold for $g=2$.
That is, some lower bound on the genus is necessary. 
We prove \autoref{theorem:main-thm-1} when the genus $g$ is bounded below by a
function depending on the
maximal dimension of any irreducible $H$-representation. 
It remains possible,
however, that the
conclusion of 
\autoref{theorem:main-thm-1}
holds whenever $g>2$, independent of $H$.
As explained in the previous paragraph, if the conclusion of \autoref{theorem:main-thm-1} were to
hold whenever $g > 2$, one would obtain the Putman-Wieland conjecture as a
consequence.

It is also unsurprising that the proof of \autoref{theorem:main-thm-1} builds
on the techniques used by the first two authors to prove the
Putman-Wieland conjecture for $g$ sufficiently large in terms of $H$
\cite[Theorem 7.2.1]{landesmanL:canonical-representations}.
\end{remark}

\begin{remark}
	\label{remark:}
	It is reasonable to hope that an even stronger statement, describing the exact image of a finite-index subgroup of $\on{Mod}_{g,n+1}$ and not its Zariski closure, might be true, but a proof of this would require additional ideas. In particular, it is natural to ask if the image is an arithmetic subgroup of its Zariski closure.
\end{remark}


In proving \autoref{theorem:main-thm-1}, it is natural to decompose $H^1(\Sigma_{g'}, \mathbb{C})$ into isotypic components corresponding to different irreducible representations of $H$. We introduce some notation to describe these:

Let $G$ be a group and $\rho: G\to \on{GL}_r(\mathbb{C})$ an irreducible representation of $G$. If $\rho$ is self-dual, then by Schur's lemma, $(\rho\otimes \rho)^G$ is one-dimensional. As $\rho\otimes \rho=\on{Sym}^2(\rho)\oplus \wedge^2\rho$, exactly one of $(\on{Sym}^2(\rho))^G, (\wedge^2\rho)^G$ is non-zero.
\begin{definition}
	Let $\rho$ be an irreducible self-dual finite-dimensional representation of a group $G$. If $(\on{Sym}^2(\rho))^G\neq 0$, we say $\rho$ is \emph{orthogonally self-dual}. If $(\wedge^2\rho)^G\neq 0$ we say $\rho$ is \emph{symplectically self-dual}.
\end{definition}
If $\mathbb V^\rho$ is a unitary local system on $\Sigma_{g,n}$ corresponding to a
representation $\rho$ of $\pi_1(\Sigma_{g,n},x)$, we define the
\emph{weight one} piece  $W_1H^1(\Sigma_{g,n}, \mathbb V^\rho)\subset H^1(\Sigma_{g,n},
\mathbb V^\rho)$ to be $$W_1H^1(\Sigma_{g,n}, \mathbb V^\rho):=H^1(\Sigma_g,
j_*\mathbb V^\rho),$$ where $j:
\Sigma_{g,n}\hookrightarrow \Sigma_g$ is the natural inclusion.
Note that this
agrees with the usual notion of weights in algebraic geometry. The groups
$W_1H^1(\Sigma_{g,n}, \mathbb V^\rho)$ and $W_1H^1(\Sigma_{g,n},
\mathbb V^{\rho^\vee})$ are naturally
dual. In particular, if $\rho$ is self-dual, $W_1H^1(\Sigma_{g,n}, \mathbb
V^\rho)$ carries a natural perfect pairing with itself. By the
graded-commutativity of the cup product, $W_1H^1(\Sigma_{g,n}, \mathbb V^\rho)$ is symplectically self-dual if $\rho$ is orthogonally self-dual, and orthogonally self-dual if $\rho$ is symplectically self-dual.

If $\rho$ factors through a surjection $\varphi: \pi_1(\Sigma_{g,n},
x)\twoheadrightarrow H$, with $H$ a finite group, there is a natural isomorphism
$\rho^\sigma\overset{\sim}{\to}\rho$ for each $\sigma\in \on{Mod}_\varphi$, the
stabilizer of $\varphi$ in $\on{Mod}_{g,n+1}$. Hence $\on{Mod}_\varphi$ acts
naturally on $W_1H^1(\Sigma_{g,n}, \mathbb V^\rho)$, via a homomorphism we name $R_{\varphi, \rho}$. 
A large part of the proof of \autoref{theorem:main-thm-1} consists of checking:
\begin{theorem}\label{theorem:main-thm-2}
Let $H$ be a finite group and let $\varphi: \pi_1(\Sigma_{g,n}, x)\twoheadrightarrow H$ be a surjective homomorphism, where $x$ is a base point of $\Sigma_{g,n}$. Suppose 
$\rho: H\to \on{GL}_r(\mathbb{C})$ is an irreducible
representation of dimension $r$, with $r$ satisfying either
\begin{enumerate}
 \item $n=0$ and $g\geq
2r+2$, or
\item $n$ is arbitrary and $g> \max(2r+1, r^2).$
 \end{enumerate} Let $\mathbb V^\rho$ denote the local system on
 $\Sigma_{g,n}$ associated to $\varphi \circ \rho$. Let $\on{Mod}_\varphi$ be the stabilizer of $\varphi$ inside
 $\on{Mod}_{g,n+1}$ and $$R_{\varphi,\rho} \colon \on{Mod}_\varphi\to
 \on{GL}(W_1H^1(\Sigma_{g,n}, \mathbb V^\rho))$$ the natural homomorphism.
 
 Then the image of $R_{\varphi,\rho}$ is Zariski-dense in 
\begin{enumerate}
\item $\on{SO}(W_1H^1(\Sigma_{g,n}, \mathbb V^\rho))$ if $\rho$ is symplectically self-dual,
\item $\on{Sp}(W_1H^1(\Sigma_{g,n}, \mathbb V^\rho))$ if $\rho$ is orthogonally self-dual, and
\item the product of $\on{SL}(W_1H^1(\Sigma_{g,n}, \mathbb V^\rho))$ with a finite
	subgroup of the center of $\on{GL}(W_1H^1(\Sigma_{g,n}, \mathbb V^\rho))$
	if $\rho$ is not self-dual.
\end{enumerate}
\end{theorem}

We prove \autoref{theorem:main-thm-2} in \autoref{subsection:main-proofs}. 

One may draw a number of concrete corollaries of \autoref{theorem:main-thm-1} and \autoref{theorem:main-thm-2}. 

\begin{corollary}\label{corollary:mumford-tate}
	Let $H$ be a finite group and $X$ a very general $H$-curve. Let $\overline{r}$ be the maximal dimension of an irreducible representation of $H$. Suppose either \begin{enumerate}
 \item $H$ acts freely on $X$ and the genus of $X/H$ is at least
$2\overline{r}+2$, or
\item the genus of $X/H$ is greater than $\max(2\overline{r}+1,\overline{r}^2).$
 \end{enumerate} Then the Mumford-Tate group of $H^1(X, \mathbb{Q})$ contains the commutator subgroup of $\on{Sp}(H^1(X, \mathbb{Q}))^H$ and is contained in $\on{GSp}(H^1(X, \mathbb{Q}))^H$.
\end{corollary}

\begin{corollary}\label{corollary:jacobian-aut}
	Let $X$ be as in \autoref{corollary:mumford-tate}. Then the endomorphism algebra of $\on{Jac}(X)$ is $\mathbb{Q}[H]$.
\end{corollary}
These corollaries are proven in \autoref{subsec:proof-of-cors}. We expect Corollary \ref{corollary:jacobian-aut} will likely have applications in equivariant birational geometry, generalizing the applications in \cite[\S7]{HassettTschinkel2022} of 
results of \cite{grunewald2015arithmetic}.

We conclude the paper with a result on the monodromy of certain special Kodaira fibrations, that is, surfaces with a smooth projective map to
a smooth curve. We refer to these as \emph{Kodaira-Parshin fibrations} (see
\autoref{definition:kp-fibration}); loosely speaking, these Kodaira-Parshin fibrations parameterize
families of covers of a fixed curve, with a moving branch point.  See
\autoref{theorem:kodaira-monondromy} for a precise algebraic statement.
These results have a purely topological interpretation---namely, they say that
for the representations considered in \autoref{theorem:main-thm-1}, their
restrictions to certain ``point-pushing'' subgroups have large image; see
\autoref{corollary:point-pushing-big-image} for a precise statement.

\subsection{The large $n$ regime}

Another regime in which we expect big monodromy is the case that we fix $g$ and
let $n$ grow large.
Specifically, we conjecture the analog of \autoref{theorem:main-thm-1} in this
context.

\begin{conjecture}\label{conjecture:large-n}
Let $H$ be a finite group and let $\Sigma_{g',n'}\to \Sigma_{g,n}$ be the
$H$-cover associated to a homomorphism $\varphi: \pi_1(\Sigma_{g,n}, x)\twoheadrightarrow H$ where $x$ is a base point of $\Sigma_{g,n}$. 
Let $\rho: H\to \on{GL}_r(\mathbb{C})$ be an irreducible $H$-representation and
$\mathbb V^\rho$ the local system associated to $\rho \circ \varphi$. 
We conjecture there is a function $c(g,\dim \rho)$ with the following property.
Suppose there are $\Delta > c(g,\dim \rho)$ points of $\Sigma_g - \Sigma_{g,n}$ so that
a small loop around each of these points is sent to a non-identity matrix under the
composition
$\pi_1(\Sigma_{g,n}) \xrightarrow{\varphi} H \xrightarrow{\rho} \on{GL}_r(\mathbb C)$.
Then, the image of the stabilizer $\on{Mod}_\varphi$ of $\varphi$ in
$\on{Mod}_{g,n+1}$ inside $\on{GL}(W_1H^1(\Sigma_{g,n}, \mathbb V^\rho))$ is Zariski-dense in 
\begin{enumerate}
\item $\on{SO}(W_1H^1(\Sigma_{g,n}, \mathbb V^\rho))$ if $\rho$ is symplectically self-dual,
\item $\on{Sp}(W_1H^1(\Sigma_{g,n}, \mathbb V^\rho))$ if $\rho$ is orthogonally self-dual, and
\item the product of $\on{SL}(W_1H^1(\Sigma_{g,n}, \mathbb V^\rho))$ 
with a finite subgroup of the center of $\on{GL}(W_1H^1(\Sigma_{g,n}, \mathbb V^\rho))$
if $\rho$ is not self-dual.\end{enumerate}
\end{conjecture}

\begin{remark}[Motivation from arithmetic statistics]
	\label{remark:arithmetic-statistics}
A number of works in arithmetic statistics over function fields have proven
results in a large $q$ limit setting by computing relevant monodromy groups with
finite coefficients
associated to spaces of $H$-covers.
See
\cite{Achter:cohenQuadratic,ellenbergVW:cohen-lenstra,fengLR:geometric-distribution-of-selmer-groups,parkW:average-selmer-rank-in-quadratic-twist-families,ellenbergL:homological-stability-for-generalized-hurwitz-spaces}
for a few examples of this; the last three references are connected to $H$-covers for a
particular group $H$ via \cite[Proposition 6.4.5]{ellenbergL:homological-stability-for-generalized-hurwitz-spaces}.

Verifying \autoref{conjecture:large-n} would give some evidence for analogous
conjectures in number theory, as it would suggest a similar big monodromy
result should be true with finite coefficients.
\end{remark}

As evidence for \autoref{conjecture:large-n}, we prove the following implication of
the conjecture. If the
monodromy is Zariski dense in the subgroups listed in \autoref{conjecture:large-n}, then it is not contained in any nontrivial parabolic,
and so in particular it does not fix any vectors.

\begin{theorem}
	\label{theorem:large-n}
Let $H$ be a finite group, $\Sigma_{g',n'}\to \Sigma_{g,n}$  an
$H$-cover, and  $\rho: H\to \on{GL}_r(\mathbb{C})$  an irreducible $H$-representation. 
Suppose there are $$\Delta > \nbound$$ points of $\Sigma_g - \Sigma_{g,n}$ so that
a small loop around each of these points is sent to a non-identity matrix under the
composition
$\pi_1(\Sigma_{g,n}) \xrightarrow{\varphi} H \xrightarrow{\rho} \on{GL}_r(\mathbb C)$.
Then, setting $\on{Mod}_\varphi\subset \on{Mod}_{g,n+1}$ to be the stabilizer of $\varphi$, there are no non-zero vectors with finite orbit under the image of $\on{Mod}_{\varphi}$ in $\on{GL}(W_1H^1(\Sigma_{g,n},
\mathbb V^\rho))$,
for $\mathbb V^\rho$ the local system associated to $\rho \circ \varphi$. 
\end{theorem}
We prove \autoref{theorem:large-n} in \autoref{subsection:large-n}.

\begin{remark}
	\label{remark:}
	\autoref{theorem:large-n} verifies new cases of the Putman-Wieland
	conjecture, \cite[Conjecture 1.2]{putmanW:abelian-quotients}.
	See also \cite[\S1.8.4]{landesmanL:canonical-representations}
	for a summary of other known cases of the Putman-Wieland conjecture.
	Note that
	\cite[Theorem 1.4.2]{landesmanL:canonical-representations} gives a
	variant of \autoref{theorem:large-n} when $g$ is large relative to
	$r$, in comparison to
\autoref{theorem:large-n},
	where we think of
	$\Delta$ as being large relative to $g$ and $r$.
\end{remark}
\begin{remark}\label{remark:large-n-expectation}
We have opted to write \autoref{theorem:large-n} with a bound on $\Delta$ depending only on the rank $r$ of our given representation $\rho$ and on the genus $g$ of $\Sigma_{g,n}$. We expect, however, that our methods could give stronger bounds in terms of the eigenvalues of the local monodromy of $\rho\circ \varphi$.
\end{remark}

\subsection{Previous work}
There is a great deal of past work related to big monodromy for  higher
Prym representations.
In the case that $H$ is abelian, \autoref{theorem:main-thm-1} 
follows from 
\cite[Corollary 2.6]{looijenga:prym-representations}.
Higher Prym representations corresponding to certain non-abelian covering groups $H$ were considered in
\cite{grunewald2015arithmetic},
and shown to give a rich class of representations of mapping class groups under certain
conditions, (namely when the cover is ``$\phi$-redundant,'') see
\cite[Theorem 1.2 and 1.6]{grunewald2015arithmetic}.
A variant for free groups was previously considered in
\cite{grunewald2009linear}. See also the recent paper \cite{looijenga2021arithmetic} for a criterion for big monodromy along somewhat different lines.
In the four papers above, the representations above were in fact shown to have \emph{arithmetic image}, meaning that
they have finite index in the integral points of their Zariski closure. Our techniques seem not to be able to establish anything towards arithmeticity. There are known examples of families of cyclic branched covers of genus zero curves with non-arithmetic monodromy \cite{deligne-mostow}.

Further arithmeticity results associated to Prym representations, primarily in
the case
that $H$ is abelian, were given in
\cite{mcmullen:braid-groups},
\cite{venkataramana:image-of-the-burau-representation},
and
\cite{venkataramana:monodromy-of-cyclic-coverings}.
The paper
\cite{salterT:arithmeticity-of-the-monodromy} proves arithmeticity of monodromy groups of certain Kodaira fibrations, corresponding to $H$-covers with $H$ a finite Heisenberg group.
In a more arithmetic direction, big monodromy associated to certain Prym
representations played a crucial role in the recent proof of Faltings' theorem
given in
\cite{lawrenceV:diophantine-problems}, see
\cite[Theorem 8.1]{lawrenceV:diophantine-problems}.

In the setting of covers of projective lines, i.e.~in a $g=0$ analogue of the
setting of \autoref{theorem:main-thm-1}, a big mod $\ell$ monodromy result was
proven by Jain \cite[Theorem 5.4.2]{jain:big-mod-ell-monodromy} when $H$ has
trivial Schur multiplier for $\ell\nmid 2|H|$. Big mod $\ell$ monodromy can
often be used to prove that the $\ell$-adic closure of the image of the mapping
class group is large, which implies big Zariski closure. A key tool is a result of
Conway-Parker (see e.g. \cite[Appendix]{fried-volklein}, 
as well as \cite[Proposition
3.4]{ellenbergVW:cohen-lenstra} and \cite{wood:an-algebraic-lifting-invariant})
which gives a stabilization of the braid group action on finite quotients
of $\pi_1(\Sigma_{g,n})$.
An analogue of this tool in the higher genus case is work of
Dunfield-Thurston 
\cite[Proposition 6.16]{dunfield-thurston},
which shows a stabilization in $g$ of the mapping class group action
on finite quotients of $\pi_1(\Sigma_g)$ as $g$ grows. The paper \cite{samperton:schur-type-invariants} gives a partial analogue of these results for the mapping class group action on
$\pi_1(\Sigma_{g,n})$ with $n > 0$. 

It may be possible to use \cite[Proposition 6.16]{dunfield-thurston} to prove a higher genus analogue of \cite[Theorem 5.4.2]{jain:big-mod-ell-monodromy}. Such a result would have the advantage over \autoref{theorem:main-thm-1} that it would give more precise information on the $\ell$-adic image of the mapping class group, but the disadvantage that the required lower bound on the genus would be ineffective.
Partial results in this direction were obtained in recent work of Sawin-Wood
\cite{sawinW:finite-quotients-of-3-manifold-groups} in the course of studying Heegard splittings of 3-manifolds.
This work uses \cite[Proposition 6.16]{dunfield-thurston} to compute the intersections of maximal isotropic subspaces with their translates by random
elements drawn from the monodromy group of Prym representations, though falls
short of computing the actual monodromy group.

Finally, recent work of
Landesman-Litt \cite{landesmanL:canonical-representations, landesmanL:applications-putman-wieland, landesmanL:introduction-putman-wieland}
shows that the monodromy group for higher Prym representations is not too small,
in the sense that it has no nonzero finite orbit vectors. Here, we build on the
methods developed in those papers to prove the monodromy group is as big as
possible. 

We next describe the primary new ideas of this work that do not appear in \cite{landesmanL:canonical-representations, landesmanL:applications-putman-wieland, landesmanL:introduction-putman-wieland}.

\subsection{Innovations of the proof}
To explain the main new ideas going into our paper, we begin with a sketch of the proof of \autoref{theorem:main-thm-2}. 

\subsubsection{Setup for the proof}
Consider a family of $n$-pointed curves $\pi: \mathscr{C}\to \mathscr{M}$, with associated family of punctured curves $\pi^\circ: \mathscr{C}^\circ\to \mathscr{M}$, so that the induced map $\mathscr{M}\to \mathscr{M}_{g,n}$ is dominant \'etale. Let $\mathbb{V}$ be a complex local system on $\mathscr{C}^\circ$ with finite monodromy, whose restriction to a fiber of $\pi^\circ$ has monodromy given by $\rho$. It suffices to compute the connected monodromy group of $W_1R^1\pi^\circ_*\mathbb{V}$, as the monodromy representation on this local system factors through the representation we are interested in. The main idea of the proof is to analyze the derivative of the period map associated to $W_1R^1\pi^\circ_*\mathbb{V}$.

\subsubsection{A novel technique: functorial reconstruction}
Using techniques building on those developed in
\cite{landesmanL:geometric-local-systems, landesmanL:canonical-representations}
we show (in \autoref{thm:vector-bundle-reconstruction} and \autoref{proposition:functoriality}) that (given our assumptions on $g$ and $\dim \mathbb{V}$) that the monodromy representation $\rho$ can be \emph{functorially reconstructed} from the derivative of the period map associated to $W_1R^1\pi^\circ_*\mathbb{V}$ at a generic point of $\mathscr{M}$ along so-called Schiffer variations. We think of this reconstruction as a new kind of ``generic Torelli theorem with coefficients'' --- this is our main technical tool. 

This enables a novel strategy to obtain information about the local system $W_1 R^1\pi^\circ_*\mathbb{V}$. We first assume for contradiction that the monodromy group of $W_1 R^1\pi^\circ_*\mathbb{V}$ has some undesirable property. Using the general theory of variations of Hodge structures, we describe the consequences of this property for the variation of Hodge structures, and in particular for the derivative of its period map. We then examine the consequences those properties have on the local system obtained by applying our functorial reconstruction algorithm, and finally show these contradict known properties of $\mathbb{V}$. 

\subsubsection{Proving \autoref{theorem:main-thm-2} by repeatedly applying
functorial reconstruction}
For example, to show $W_1 R^1\pi^\circ_*\mathbb{V}$ is irreducible as a representation of the monodromy group, we assume for contradiction it is reducible. It would be convenient if this implies that $W_1 R^1\pi^\circ_*\mathbb{V}$ is reducible as a representation of Hodge structures, but this is not the case -- it could instead be, for example, the tensor product of a fixed irreducible Hodge structure with an irreducible variation of Hodge structure. However, in this case the derivative of the period map is still reducible in a suitable sense, and in fact one can check this holds for any variation of Hodge structures whose underlying local system is reducible. Applying the reconstruction algorithm to a reducible derivative of the period map, we obtain a reducible local system, contradicting the irreducibility of $\mathbb{V}$.

A slight enhancement of this argument, namely \autoref{theorem:simplicity-of-VHS}, shows that the connected monodromy group of $W_1R^1\pi^\circ_*\mathbb{V}$ is in fact a \emph{simple}
group, acting irreducibly. By the classification of simple factors of Mumford-Tate groups of Abelian varieties, this
group necessarily acts through a minuscule representation,
see \cite[Theorem 0.5.1(b)]{zarhin1984weights}.

We rule out nonstandard representations in \autoref{section:main-proofs} via the same strategy. We show that, for these representations, the rank of the derivative of the period map along a Schiffer variation is necessarily large. Applying the reconstruction algorithm gives a local system of large rank -- in fact, larger than the rank of $\mathbb{V}$, leading to a contradiction. 

The last case remaining to rule out is when the monodromy group of $W_1 R^1\pi^\circ_*\mathbb{V}$ is a classical group but $\mathbb{V}$ is not self-dual. In this case we show that  $W_1 R^1\pi^\circ_*\mathbb{V}$ is self-dual as a variation of Hodge structures, and deduce from the reconstruction algorithm that $\mathbb V$ is self-dual, obtaining a contradiction.
This concludes our sketch of the idea of the proof of
\autoref{theorem:main-thm-2}.

\subsubsection{Proving \autoref{theorem:main-thm-1}}
To deduce \autoref{theorem:main-thm-1}, one may explicitly describe the
commutator subgroup of $\on{Sp}(H^1(\Sigma_{g'}, \mathbb{C}))^H$ as a product of
the groups appearing in \autoref{theorem:main-thm-2} --- that theorem implies
that the connected monodromy group of $H^1(\Sigma_{g'}, \mathbb{C})$ surjects
onto each of the simple factors of $\on{Sp}(H^1(\Sigma_{g'}, \mathbb{C}))^H$. It
then suffices by the Goursat-Kolchin-Ribet criterion of Katz \cite[Proposition
1.8.2]{katz-esde} to show that for $\rho_1, \rho_2$ irreducible
$H$-representations, with associated local systems $\mathbb{U}_1, \mathbb{U}_2$
on $\mathscr{C}^\circ$, an isomorphism $W_1R^1\pi^\circ_*\mathbb{U}_1\simeq
W_1R^1\pi^\circ_*\mathbb{U}_2$ necessarily comes from an isomorphism
$\rho_1\simeq\rho_2$. We deduce this from our functorial reconstruction
results.

\subsubsection{A new technical ingredient:
global generation}

Although the heart of the our functorial reconstruction results rest on understanding the derivative of a
certain period map, following similar techniques introduced in \cite{landesmanL:applications-putman-wieland,
landesmanL:canonical-representations}, there are a number of substantial 
innovations.
In particular, the key to the proof of \autoref{thm:vector-bundle-reconstruction} is \autoref{proposition:globally-generated}, which analyzes the global generation properties of flat vector bundles under isomonodromic deformation.
In past work, the first two authors were only able to prove that the relevant
vector bundles were {\em generically} globally generated, as opposed to actually
being globally generated. By analyzing the
obstruction for generically globally generated bundles to be globally
generated, we are able to push the methods developed previously farther.

The proof of our main results
underlies a new connection between \emph{global generation} of certain vector
bundles on curves and questions about big monodromy (see
\autoref{section:questions} for further details). As far as we are aware this
connection is totally novel.
%

\subsection{Acknowledgments}
Landesman was supported by
the National Science
Foundation 
under Award No.
DMS 2102955.
Litt was supported by the NSERC Discovery Grant, ``Anabelian methods in arithmetic and
algebraic geometry.''  Sawin was supported by NSF Grant DMS-2101491 and a Sloan Research Fellowship. The authors are grateful for useful discussions with Kevin Chang, Josh Lam, Eduard Looijenga, Alex Lubotzky, Andrew Putman, Kasra Rafi, Andy Ramirez-Cote, and Bena Tshishiku.

\section{Notation and preliminaries on moduli}\label{section:notation-and-moduli}
\subsection{Notation}\label{notation:versal-family-of-covers}
	Throughout this paper, we work over the complex numbers.
	Suppose we are given a smooth proper family of curves $\pi: \mathscr C \to
	\mathscr M$ with geometrically connected fibers and $n$ sections $s_1, \ldots, s_n: \mathscr M \to \mathscr
	C$ with disjoint images
	$\mathscr{D}_1,\cdots,\mathscr{D}_n \subset \mathscr C$. 
	Let $\mathscr{C}^\circ=\mathscr{C} - \{\mathscr{D}_1, \cdots,
	\mathscr{D}_n\}$.
	If the induced map $\mathscr M \to \mathscr M_{g,n}$ is dominant
	\'etale, we say $\pi$ is a \emph{versal family of $n$-pointed curves of genus $g$}, and refer to $\pi^\circ := \pi|_{\mathscr C^\circ}$ as the associated \emph{versal family of $n$-times punctured curves}. 

	Suppose, moreover, we have
	a diagram 
	\begin{equation}
	\xymatrix{
		\mathscr{X} \ar[r]^f \ar[rd]_{\pi'}& \mathscr{C} \ar[d]_\pi\\
	&  \mathscr{M}, \ar@/_/[u]_{s_1, \cdots, s_n} }
		\label{equation:versal-diagram}
	\end{equation}
	where $\pi$ is a versal family of $n$-pointed curves of genus $g$,
	and $\pi'$ is a smooth proper curve with geometrically connected fibers of
	genus $g'$.
	Suppose $f$ is finite, Galois, and unramified away from $\cup_{i=1}^n \mathscr{D}_i$. Let $\mathscr{X}^\circ=f^{-1}(\mathscr{C}^\circ),$ let $\pi^\circ :=
	\pi|_{\mathscr C^\circ}$, and let $f^\circ=f|_{\mathscr{X}^\circ}$. 
	For $m\in \mathscr{M}$  a point, we set $C=\mathscr{C}_m$, $C^\circ
	=\mathscr{C}^\circ_m$, and $D=C - C^\circ$; let $c\in C^\circ$ be a point. 
	As $f^\circ$ is finite Galois, it induces a surjection $
	\pi_1(\mathscr{C}^\circ,c)\twoheadrightarrow H$ for some finite group
	$H$. As $\pi'$ has geometrically connected fibers and $f$ is Galois, the composition $$\psi: \pi_1(C^\circ, c)\to \pi_1(\mathscr{C}^\circ, c)\to H$$ is surjective. 
	
	Let $x\in \Sigma_{g,n}$ be a basepoint and $\varphi: \pi_1(\Sigma_{g,n}, x)\twoheadrightarrow H$ a surjection.
	A \emph{ versal family of $\varphi$-covers} is the data of a diagram as in \eqref{equation:versal-diagram} above, together with
	\begin{enumerate}
\item a point $c\in \mathscr{C}^\circ$, $m=\pi^\circ(c)$, $C^\circ = \mathscr
	C^\circ_m$, and an identification $i: (\Sigma_{g, n}, x)\simeq
(C^\circ, c)$, such that
\item under this identification the map $\psi$ above identifies with $\varphi$.
\end{enumerate}

If $\rho$ is a representation of a finite group $H$, $\rho: H \to \gl_r(\mathbb C)$ 
and $\varphi: \pi_1(\Sigma_{g,n},x) \to H$ is a map, we often use $\mathbb
V^\rho$ to denote the local system on $\Sigma_{g,n}$ associated to $\rho \circ
\varphi$.

\begin{remark}
	\label{remark:}
	Note that versal families of $\varphi$-covers exist by
\cite[Theorem 4]{wewers:thesis}.
One may also construct such a versal family by taking an open substack of the
stack
$\mathscr B_{g,n}(H)$ as constructed in
\cite[\S2.2]{abramovichCV:twisted-bundles} (where the group $H$ here is called
$G$ there).
The above constructions give versal families of Deligne-Mumford stacks.
Hence, if one wishes, one may pass to a dominant \'etale cover of the $\mathscr{M}$ thus constructed, making all the objects in question schemes.
\end{remark}
\subsection{Preliminaries on moduli}\label{subsection:moduli-preliminaries}
In this section we explain how to interpret the main theorems of this paper as being about the monodromy of (summands of) $R^1\pi'_*\mathbb{C}$, for $\pi'$ as in \eqref{equation:versal-diagram}, arising from a versal family of $H$-covers.

Recall that, given a cover $\Sigma_{g', n'}\to \Sigma_{g,n}$, we are studying
the action of finite index subgroups of $\on{Mod}_{g,n+1}$ on the abelianization
of $\pi_1(\Sigma_{g', n'}, x')$, a finite index subgroup of $\pi_1(\Sigma_{g,n},x)$. 
One may interpret the action of $\on{Mod}_{g,n+1}$ on $\pi_1(\Sigma_{g,n}, x)$ as follows. Let $\mathscr{M}_{g,n}$ be the Deligne-Mumford moduli stack of genus $g$ curves with $n$ marked points,  $\mathscr{C}_{g,n}/\mathscr{M}_{g,n}$ be the universal family, and $\mathscr{C}_{g,n}^{\circ}$ the associated family of $n$-times punctured curves. Note that $\mathscr{C}_{g,n}^\circ$ is canonically isomorphic to $\mathscr{M}_{g,n+1}$. 

Now consider the map $$p_1: \mathscr{C}_{g,n}^\circ \times_{\mathscr{M}_{g,n}}\mathscr{C}_{g,n}^\circ\to \mathscr{C}_{g,n}^\circ$$ given by projection onto the first factor. This map has a canonical section given by the diagonal $\Delta$. Thus there is a  short exact sequence of fundamental groups $$1\to \pi_1(\Sigma_{g,n})\to \pi_1(\mathscr{C}_{g,n}^\circ \times_{\mathscr{M}_{g,n}}\mathscr{C}_{g,n}^\circ)\overset{p_{1*}}{\longrightarrow} \pi_1(\mathscr{C}_{g,n}^\circ)\to 1,$$ split by $\Delta_*$,
inducing a natural action of $\pi_1(\mathscr{C}_{g,n}^\circ)$ on $\pi_1(\Sigma_{g,n})$. It follows from e.g.~\cite[\S10.6.3 and p.~353]{farbM:a-primer} that there is a natural identification of $\pi_1(\mathscr{C}_{g,n}^\circ)=\pi_1(\mathscr{M}_{g,n+1})$ with $\on{Mod}_{g,n+1}$, under which this action identifies with the one we are studying.

Fix a surjection $\varphi: \pi_1(\Sigma_{g,n},x)\twoheadrightarrow H$ and a versal family of $\varphi$-covers with notation as in \eqref{equation:versal-diagram}. In analogy with the fiber product construction above, we may consider the map $$q_1: \mathscr{C}^\circ\times_{\mathscr{M}} \mathscr{X}\to \mathscr{C}^\circ$$ given by projection onto the first coordinate. We claim that the monodromy of $R^1q_{1*}\mathbb{C}$ factors through the representation $R_\varphi$ studied in \autoref{theorem:main-thm-1}, and indeed factors through a finite index subgroup of $\on{Mod}_{\varphi}$. 

We first construct a map from $\pi_1(\mathscr{C}^\circ)$ to $\on{Mod}_\varphi$. The family of $n+1$-pointed curves $$\mathscr{C}^\circ\times_{\mathscr{M}} \mathscr{C}\to \mathscr{C}^{\circ}$$ (with sections given by the pullbacks of $s_1, \cdots, s_n$ and the tautological section induced by the diagonal) induces a map $\mathscr{C}^\circ \to \mathscr{M}_{g,n+1}$. 
By the discussion above and \cite[Lemma 2.1.4]{landesmanL:canonical-representations}, the image of the induced map on fundamental groups $\pi_1(\mathscr{C}^\circ)\to \pi_1(\mathscr{M}_{g,n+1})$
is finite index in $\pi_1(\mathscr{M}_{g,n+1})=\on{Mod}_{g,n+1}$. 
By the definition of $\mathscr M$, the image of the induced map is contained in
$\on{Mod}_{\varphi}$, 
as defined as in \autoref{theorem:main-thm-1},
and hence the image has finite index in $\on{Mod}_{\varphi}$.

Unwinding this discussion, the monodromy representation on $R^1q_{1*}\mathbb{C}$
factors through the representation $R_{\varphi}$, with image a finite index
subgroup of the image of $R_{\varphi}$. Thus to understand the image of
$R_\varphi$, it suffices to understand the image of this monodromy
representation associated to $R^1q_{1*}\mathbb{C}$.

 Moreover $q_1$ is itself the pullback of $\pi'$ along $\pi^\circ$, which
 induces a surjection on fundamental groups. Thus the image of the monodromy
 representation associated to
$R^1q_{1*}\mathbb{C}$
is the same 
as that of the monodromy representation associated to $R^1\pi'_*\mathbb{C}$. 
 
 Now let $\rho: H\to GL_r(\mathbb{C})$ be a representation and $\mathbb{V}^\rho$
 the associated local system on $\mathscr{C}^\circ$. A completely analogous
 argument shows that the image of the monodromy representation associated to
 $W_1R^1\pi^\circ_*\mathbb{V}^\rho$ is the same up to finite index as the representation $R_{\varphi, \rho}$ considered in \autoref{theorem:main-thm-2}. 

\section{Review of parabolic bundles and period maps}
\label{section:parabolic-review}
Let $C$ be a smooth projective curve over $\mathbb{C}$, and $D=x_1+\cdots+x_n$ a reduced effective divisor on $C$. 
\begin{definition}
	\label{definition:parabolic-vecor-bundle}
	A \emph{parabolic vector bundle} $E_\star$ on $(C,D)$ is a vector bundle $E$ on $C$, a
decreasing filtration $E_{x_j} = E_j^1 \supsetneq E_j^2 \supsetneq \cdots
\supsetneq E_j^{n_j+1} =
0$ for each $1 \leq j \leq n$, and an increasing sequence of real numbers $0\leq
\alpha_j^1<\alpha_j^2<\cdots<\alpha_j^{n_j}<1$ for each $1 \leq j \leq n$,
referred to as {\em weights}. 
Here, $E_{x_j}$ refers to the fiber of $E$ at $x_j$, and hence the filtration is
merely a filtration of vector spaces on the fiber of $E$ at $x_j$, not a filtration of
the vector bundle $E$.
We use $E_\star = (E, \{E^i_j\}, \{\alpha^i_j\})$ to denote the data of a parabolic bundle. Given a parabolic bundle $E_\star$, we will often write $E_0$ for the underlying vector bundle $E$.
\end{definition}

\subsection{}
\label{subsection:parabolic-definitions}
Parabolic
bundles admit a notion of \emph{parabolic stability}, analogous to the usual
notion of stability for vector bundles, which we next recall. 
First, the {\em parabolic degree} of a parabolic bundle $E_\star$ is
\begin{align*}
	\on{deg}(E_\star) := \deg(E) + \sum_{j=1}^n \sum_{i=1}^{n_j}
	\alpha^i_j \dim(E_j^i/E_j^{i+1}).
\end{align*}
Then, the {\em parabolic slope} is defined by $\mu_\star(E_\star) :=
\on{deg}(E_\star)/\rk(E_\star)$.
Any subbundle $F \subset E$ has an induced parabolic structure $F_\star \subset
E_\star$ defined as
follows:
we take the filtration
over $x_j$ on $F$ is obtained
	from the filtration $$F_{x_j} = E^1_j \cap F_{x_j} \supset E^2_j \cap
	F_{x_j} \supset \cdots \supset E^{n_j+1}_j \cap F_{x_j} = 0$$
	by removing redundancies. For the weight associated to $F^i_j \subset
	F_{x_j}$ one takes $$\max_{\substack{k, 1 \leq k \leq n_j}} \{ \alpha^k_j : F^i_j = E^k_j \cap
	F_{x_j}\}.$$
A parabolic bundle $E_\star$ is {\em parabolically semi-stable} if for every
nonzero subbundle $F \subset E$ with
induced parabolic structure $F_\star$, we have $\mu_\star(F_\star) \leq
\mu_\star(E_\star)$.
Similarly, a 
parabolic bundle $E_\star$ is {\em parabolically stable} if for every
nonzero subbundle $F \subset E$ with
induced parabolic structure $F_\star$, we have $\mu_\star(F_\star) <
\mu_\star(E_\star)$.

We next introduce the notation $E^\rho_\star$ for the parabolic bundle
corresponding to a unitary representation $\rho$.
\begin{notation}
	\label{notation:rep-to-vector-bundle}
	Let $C$ be a curve and $D \subset C$ a divisor.
Recall that under the Mehta-Seshadri correspondence \cite{mehta1980moduli}, there is a bijection
between irreducible unitary representations of $\pi_1(C-D)$ and parabolic degree $0$ stable
parabolic vector bundles on $C$, with parabolic structure along $D$.
Given an irreducible unitary representation $$\rho: \pi_1(C-D)\to U(r)$$ of dimension $r$, we use $E^\rho_\star$ to denote
the parabolic bundle associated to $\rho$. The underlying vector bundle $E^\rho_0$ of $E^\rho_\star$ is the Deligne canonical extension of the flat bundle on $C-D$ associated to $\rho$, i.e.~$E^\rho_0$ carries a flat connection $$\nabla: E^\rho_0\to E^\rho_0\otimes \Omega^1_C(\log D)$$
with monodromy $\rho$ and whose residues have eigenvalues with real parts in $[0,1)$. See \cite[Definition 3.3.1]{landesmanL:geometric-local-systems} for details of how to associate a parabolic structure to a connection.

Such unitary representations will often arise as follows: we will start with a
surjection $\varphi: \pi_1(C-D)\twoheadrightarrow H$ for some finite group $H$.
Then for each irreducible representation $\rho$ of $H$, the representation
$\rho\circ\varphi$ is unitary, and we will abuse notation to denote it by $\rho$ as well. 

Let $D^\rho_{\on{non-triv}} \subset D$ denote the subset of points $p \in D$
so that the local inertia at $p$ under $\rho$ is not the identity.
Set $\Delta$ to be the number of points in $D^\rho_{\on{non-triv}}$.
\end{notation}

\begin{definition}
Given a parabolic bundle
$E_\star = (E, \{E^i_j\}, \{\alpha^i_j\})$,
let $J \subset \{1, \ldots, n\}$ denote the set of 
integers $j \in \{1, \ldots, n\}$ for which $\alpha^1_j = 0$, and define
\begin{align}
	\label{equation:coparabolic}
\widehat{E}_0 : = \ker( E \to \oplus_{j \in J} E_{x_j}/E_j^2).
\end{align}
(This is a special case of more general notation used for coparabolic bundles
as in \cite[2.2.8]{landesmanL:geometric-local-systems} or the equivalent
\cite[Definition 2.3]{bodenY:moduli-spaces-of-parabolic-higgs-bundles}, but is all we will need for this
paper.)
In particular, $\widehat{E}_0 \subset E$ is a subsheaf. Note that if $E_\star$ is the parabolic bundle associated to a unitary representation of $\pi_1(C-D)$, the natural logarithmic connection on $E_0$ descends to a logarithmic connection on $\widehat E_0$ (albeit with different residues), by a local calculation.
\end{definition}
\begin{proposition}\label{proposition:weight-1-hodge-filtration}\protect{}
	Let $\mathbb{V}$ be a unitary local system on $C - D$, and let
	$E_\star$ be the associated parabolic bundle on $C$. Then there is a
	natural mixed Hodge structure on $H^1(C - D, \mathbb{V})$, and natural
	isomorphisms $$(F^1\cap W_1)H^1(C - D, \mathbb{V})=H^0(C, \widehat E_0\otimes \omega_C(D))$$
	$$W_1H^1(C - D, \mathbb{V})/(F^1\cap W_1)H^1(C - D, \mathbb{V})=H^1(C, E_0).$$
\end{proposition}
\begin{proof}
Everything except the last line is 	\cite[Lemma
	3.2]{landesmanL:applications-putman-wieland}. The last line is just
	unwinding definitions, see
	e.g.~\cite[\S3]{landesmanL:applications-putman-wieland} or \cite[Theorem
	4.1.1]{landesmanL:canonical-representations}.
\end{proof}

Now suppose $\pi: \mathscr{C}\to \mathscr{M}$ is a versal family of
punctured $n$-pointed curves, $\pi^\circ:\mathscr{C}^\circ\to \mathscr{M}$ is the associated punctured family, 
and $\mathbb{V}$ a unitary local system on $\mathscr{C}^\circ$. It is well-known
(see \cite[Theorem 4.1.1]{landesmanL:canonical-representations}) that $R^1\pi^\circ_*\mathbb{V}$ carries an admissible complex variation of Hodge structures, with the fibers of $W_1R^1\pi^\circ_*\mathbb{V}$ as described in \autoref{proposition:weight-1-hodge-filtration}. Let $m\in \mathscr{M}$
 be a point, $C=\pi^{-1}(m)$, $C^\circ=(\pi^\circ)^{-1}(m)$, and $D=C - C^\circ$. Let $\rho$ be the monodromy representation associated to $\mathbb{V}|_{C^\circ}$. The derivative of the period map associated to $W_1R^1\pi^\circ_*\mathbb{V}$ is, by \autoref{proposition:weight-1-hodge-filtration}, a map $$dP_m^\rho: T_m\mathscr{M}\to \on{Hom}(H^0(C, \widehat{E}^\rho_0\otimes\omega_C(D)), H^1(C, E^\rho_0)).$$ As $\pi^\circ$ is a versal family, for each $m\in \mathscr{M}$ we have that $T^*_m\mathscr{M}=H^0(\omega_C^{\otimes 2}(D))$, and so by Serre duality we may adjointly obtain a map $$H^0(C, \widehat
	E^\rho_0\otimes \omega_C(D))\otimes H^0(C, (E^\rho_0)^\vee\otimes \omega_C)\to H^0(C, \omega_C^{\otimes 2}(D)).$$ 

	There is another natural map between these vector spaces, which we
	denote by $B^\rho_m$ and describe next.
	There is a bilinear form
	\begin{equation}\label{mini-bilinear-form}  \widehat
	E^\rho_0 \otimes  (E^\rho_0)^\vee \to  \mathcal O_C \end{equation}
	arising from the inclusion $\widehat{E}^\rho_0 \to E^\rho_0$ and the pairing $E^\rho_0 \otimes (E^\rho_0)^\vee\to \mathcal O_C$. 	 
	Twisting this bilinear form by $\omega_C^{\otimes 2}(D)$ yields the map
	of sheaves
	\begin{equation*}  \widehat
	 E^\rho_0\otimes \omega_C(D)\otimes  (E^\rho_0)^\vee\otimes \omega_C
 \to  \omega_C^{\otimes 2}(D). \end{equation*}
 	Taking global sections yields the map
	\begin{equation} 
		\label{equation:pairing-global-sections}
		H^0(C, \widehat
	E^\rho_0\otimes \omega_C(D)\otimes  (E^\rho_0)^\vee\otimes \omega_C)
\to H^0(C, \omega_C^{\otimes 2}(D)). \end{equation}
	Finally, we define
	\begin{align}
		\label{key-bilinear-form}
		B^\rho_m: H^0(C, \widehat
	E^\rho_0\otimes \omega_C(D))\otimes H^0(C, (E^\rho_0)^\vee\otimes \omega_C)\to H^0(C, \omega_C^{\otimes 2}(D))
\end{align}
	to be the composition of the
	multiplication map
	\begin{equation*} H^0(C, \widehat
	E^\rho_0\otimes \omega_C(D))\otimes H^0(C, (E^\rho_0)^\vee\otimes \omega_C)\to H^0(C, \widehat
	E^\rho_0\otimes \omega_C(D)\otimes  (E^\rho_0)^\vee\otimes \omega_C) \end{equation*}
	with \eqref{equation:pairing-global-sections}.
	
	By combining \autoref{proposition:weight-1-hodge-filtration} above with \cite[Theorem
5.1.6]{landesmanL:canonical-representations}, we obtain the following
description of the derivative of the period map.

\begin{proposition}\label{proposition:derivative-of-period-map}
	Let $\pi: \mathscr{C} \to \mathscr{M}$ be a versal family of marked
	curves of genus $g$ as in \autoref{notation:versal-family-of-covers},
	and let $\pi^\circ:\mathscr{C}^\circ\to \mathscr{M}$ be the associated
	punctured family. Fix $m\in \mathscr{M}$ and set $C=\pi^{-1}(m)$,
	$C^\circ=(\pi^\circ)^{-1}(m)$, and $D=C - C^\circ$.
	Let $\mathbb{V}$ be a unitary local system on $\mathscr{C}^\circ$, and let $\rho$ be the monodromy representation of $\mathbb{V}|_{C^\circ}$. Then
	the derivative of the period map associated to $W_1R^1\pi_\ast^\circ
	\mathbb{V}$ is identified with the map $B^\rho_m$ defined in
	\eqref{key-bilinear-form}
\end{proposition}

The map $dP_m^\rho$ described above gives rise to a number of other maps by adjointness; we will find it convenient to name some of them. 
We have denoted the bilinear pairing of \eqref{key-bilinear-form} by $B^\rho_m$. 
We will denote by $\theta^\rho_m$ the adjoint map 
\begin{align}
	\label{equation:theta-map}
\theta^\rho_m: H^0(C, \widehat{E}_0^\rho\otimes \omega_C(D))\to H^1(C,
E^\rho_0)\otimes H^0(C, \omega_C^{\otimes 2}(D)),
\end{align}
where we identify $H^1(C, E_0)$ with $H^0(C, (E^\rho_0)^\vee\otimes
\omega_C)^\vee$ by Serre duality.
\section{Global generation of vector bundles on generic curves}
\label{section:global-generation}

\subsection{A generalization of the non-ggg lemma} We will need a generalization of the non-ggg ({\bf g}enerically {\bf g}lobally {\bf g}enerated) lemma from \cite[Proposition
6.3.6]{landesmanL:geometric-local-systems} (see also \cite[Proposition
6.3.1]{landesmanL:geometric-local-systems}).
In order to prove this generalization, we will make use of the following lemma, bounding the rank of the global sections of a subbundle of a vector bundle in terms of various numerical invariants.

\begin{lemma}
	\label{lemma:ggg-criterion}
	Let $F$ be a vector bundle on a smooth proper curve $C$ of genus $g$, and let $V \subset
	F$ be a subbundle with $c := \rk F - \rk V$.
	Suppose that $0 = N^0 \subset N^1 \subset \cdots \subset N^k = V$ is the
	Harder Narasimhan filtration of $V$.
	Let $\mu$ be a rational number.
	If
	\begin{enumerate}
		\item $h^0(C, F) - h^0(C, V) = \delta$,
		\item $0\leq \mu(N^i/N^{i-1}) \leq 2g  -2$ for $1 \leq i \leq k$, and
		\item $\mu(F) \geq \mu$,
	\end{enumerate}
	then
	\begin{align*}
		\rk F(2g - 1- \mu) \geq cg - \delta.
	\end{align*}
\end{lemma}
\begin{proof}
\cite[Lemma 6.2.1]{landesmanL:geometric-local-systems} combined with Riemann-Roch for $F$ give
\begin{align*}
		\frac{ \deg V}{2} +\rk V \geq  h^0(C, V) = h^0(C, F) - \delta \geq \deg F - (g-1) \rk F -
		\delta.
	\end{align*}

	By assumption $(2)$ that $\mu(N^i/N^{i-1}) \leq 2g -2$, we obtain $\mu(V) \leq
	2g - 2$, i.e. $\deg V \leq (2g-2) \rk V$, giving
	\begin{align*}
		g \rk V  \geq \deg F - (g-1) \rk F -
		\delta.
	\end{align*}
	Then, $\rk V= \rk F-c$ and $\mu(F) \geq \mu $ gives
		\begin{align*}
			g (\rk F- c)  = g(\rk V) &\geq \deg F - (g-1) \rk F -
			\delta \\
			&\geq \mu(F) \rk F - (g-1) \rk F - \delta \\
			&\geq \mu \rk F - (g-1) \rk F - \delta.
	\end{align*}
	Simplifying this gives
	\begin{equation}
	\rk F(2g - 1- \mu) \geq cg - \delta.\qedhere
\end{equation}
\end{proof}

\begin{proposition}
	\label{proposition:non-ggg-improved}
		Let $\mu$ be a rational number. Suppose $E$ is a vector bundle on 
	a smooth proper connected genus $g$ curve and any nonzero quotient bundle $E \twoheadrightarrow Q$
	satisfies 
	$\mu(Q) \geq \mu$.
		Let $U \subset E$ be a proper subbundle with
		$c := \rk E - \rk U>0$
		and
		$\delta := h^0(C, E) - h^0(C, U)$.
		If $\mu< 2g - 1$ then we have
		$\rk E \geq \frac{c g-\delta}{2g - 1 -\mu}$ while if $\mu \geq
		2g-1$ then we have $\delta \geq cg + \mu - (2g-1)$.
\end{proposition}
\begin{proof}

	As a first step, we may reduce to the case $U$ is generically globally 
	generated
	by replacing $U$ with the saturation of the image of $H^0(C, U) \otimes
	\mathscr O_C \to U \to V$.
	Next, let $N^i$ denote the 
	largest filtered part of the Harder-Narasimhan
	filtration of $U$, whose associated graded sequence of vector bundles
	all have slopes $>2g-2$.
	Because $U \subset E$ was assumed saturated, and $N^i \subset U$ is
	saturated, the quotient $U/N^i \subset
	E/N^i$ is also a saturated inclusion of vector bundles.

	Now, let $F := E/N^i$ and let $V := U/N^i$.
	We now wish to apply \autoref{lemma:ggg-criterion} to the subbundle $V
	\subset F$, which will tell us
	\begin{equation}\label{rank-F-times} \rk F(2g - 1- \mu) \geq cg - \delta .\end{equation}
	If we also assume $2g-1>\mu$, 
	dividing both sides of \eqref{rank-F-times} by $2g-1-\mu$ gives our desired inequality
	$\rk E \geq \rk E/N^i \geq \frac{c g-\delta}{2g - 1 -\mu}$.
	
	On the other hand, if $\mu \geq 2g-1$, then rearranging
	\eqref{rank-F-times}, and using that $\rk F \geq 1$, (since $V \subset F$
is a proper subbundle,) gives the claimed inequality
	$$\delta \geq cg - \rk F(2g-1-\mu) = cg + \rk F(\mu - (2g-1)) \geq cg +
\mu - (2g - 1).$$

	To conclude, it remains to verify conditions $(1)$, $(2)$, and $(3)$ of
	\autoref{lemma:ggg-criterion}.

	For condition 
	\autoref{lemma:ggg-criterion}$(1)$, we will show 
	$$\delta = h^0(C, E) -h^0(C, U) = h^0(C, E/N^i) - h^0(C, U/N^i).$$
	Indeed, since $H^1(C, N^i) = 0$ by \cite[Lemma
	6.3.5]{landesmanL:geometric-local-systems},
	we find $h^0(C, E/N^i) = h^0(C, E) - h^0(C, N^i)$ and $h^0(C, U/N^i) =
	h^0(C, U) - h^0(C, N^i)$. This implies $h^0(C, E) -h^0(C, U) = h^0(C,
	E/N^i) - h^0(C, U/N^i)$,
	and the former is $\delta$ by definition.

	Note that condition \autoref{lemma:ggg-criterion}$(3)$
	holds because $F$ is a quotient of $E$, and so $\mu(F) \geq \mu$ by
	assumption.

	Finally, we verify \autoref{lemma:ggg-criterion}$(2)$.
	We wish to show each associated graded piece of the Harder-Narasimhan
	filtration of $V$ has slope between $0$ and $2g-2$.
	The upper bound follows from the construction of $V$ as $U/N^i$.
	We conclude by verifying the lower bound.
	Recall we assumed above that $U$ is generically globally generated.
	Therefore, $V$ is
	generically globally generated as it is a quotient of $U$, and hence any quotient of $V$ is itself generically globally generated, and thus has positive slope. Applying this to the smallest slope associated graded piece of the Harder-Narasimhan filtration 
	 verifies \autoref{lemma:ggg-criterion}$(2)$.
\end{proof}

\begin{lemma}
	\label{lemma:ggg-parabolic-consequence}
	Suppose $E_\star = (E, \{E^i_j\}, \{\alpha^i_j\})$ is a nonzero parabolic
	bundle on  $(C,D)$, where $C$ is a smooth proper connected genus $g$ curve
		and $D = x_1 + \cdots + x_n$ is a reduced effective divisor. Assume ${E}_\star$ is parabolically semistable of
		slope $\mu+n$ with $\mu < 2g -1$.
		Suppose $\widehat E_0$ has a proper subbundle $U \subset \widehat E_0$ with
		$c := \rk \widehat E_0 - \rk U$
		and
		$\delta := h^0(C, \widehat E_0) - h^0(C, U)$.
		Then $\rk E= \rk \widehat E_0 \geq \frac{c g-\delta}{2g - 1 -\mu}$.
\end{lemma}
\begin{proof}
	Note that 
	\cite[Lemma 6.3.4]{landesmanL:geometric-local-systems} implies any
	quotient of vector bundles $\widehat E_0 \twoheadrightarrow Q$ satisfies $\mu(Q)
	\geq \mu$.
	We may therefore conclude by applying \autoref{proposition:non-ggg-improved} to the vector bundle
	$\widehat{E}_0$.
\end{proof}

\begin{lemma}
	\label{lemma:dual-hat}
	The bundle $E^{\rho^\vee}_0$ is semistable if and only if $\widehat{E^\rho_0}$
	is semistable. 
	The bundle $E^{\rho^\vee}_0$ is stable if and only if $\widehat{E^\rho_0}$
	is stable. 
\end{lemma}
\begin{proof}
	Stability of $\widehat{E}^\rho_0$ is
	equivalent to
	stability of $\widehat{E}^\rho_0(D)$. In turn, this is
	equivalent to stability of $(E^{\rho^\vee}_0)^\vee$
	by \cite[(3.1)]{yokogawa:infinitesimal-deformation},
	which shows
	$\widehat{E}^\rho_0(D) \simeq (E^{\rho^\vee}_0)^\vee$.
	Finally, stability of $(E^{\rho^\vee}_0)^\vee$
	is equivalent to stability of $E^{\rho^\vee}_0$.
	The same holds with semistability in place of stability.
\end{proof}
\subsection{Global generation and deformation theory}
In this section, we deduce 
\autoref{proposition:globally-generated}
from \autoref{lemma:ggg-parabolic-consequence}
by passing to a generic curve. We noe that
\autoref{proposition:globally-generated}
is a result  about \emph{global
generation} of vector bundles, as opposed to just generic global generation.

Suppose $C$ is a curve, $D$ is a reduced effective divisor in $C$, and $E$ is a vector bundle on $C$. Below we denote by $\on{At}_{(C,D)}(E)$ the preimage of $T_C(-D)$ under the natural map $\on{At}_C(E)\to T_C$, where $\on{At}_C(E)$ is the Atiyah bundle of $E$.
For some background on Atiyah bundles relevant to the context of this paper, see
\cite[\S3]{landesmanL:geometric-local-systems}.
We begin by recalling a couple of general facts from deformation theory relating
to the Atiyah bundle.
\begin{lemma}\label{lemma:atiyah-defs}
	Let $C$ be a smooth projective curve over a field $k$, and $D$ a reduced
	effective divisor on $C$. Let $E$ be a vector bundle on $C$. Then the
	space of first-order deformations of the triple $(C,D,E)$ is naturally
	in bijection with the first cohomology of the Atiyah bundle, $H^1(C, \on{At}_{(C,D)}(E))$. 
\end{lemma}
\begin{proof}
See \cite[Proposition 3.5.5]{landesmanL:geometric-local-systems} and the references therein, or see \cite[Theorem 3.3.11]{sernesi2007deformations} for a proof in the case where $E$ is a line bundle and $D$ is empty; the general case is identical.
\end{proof}
\begin{lemma}\label{lemma:atiyah-vanishing-h1}
	With notation as in \autoref{lemma:atiyah-defs}, let $s\in\Gamma(C, E)$ be a global section. Let $$\psi: \on{At}_{(C,D)}(E)\to E$$ be the map sending a differential operator $\xi$ to $\xi(s)$. Then given $\nu\in H^1(C, \on{At}_{(C,D)}(E))$, corresponding to some first order deformation $(\mathscr{C}, \mathscr{D}, \mathscr{E})$ over $k[\epsilon]/\epsilon^2$, the section $s$ extends to a global section of $\mathscr{E}$ if and only if $$\psi(\nu)\in H^1(C, E)$$ vanishes. 
	
	In particular, if $E$ carries a flat connection $\nabla$ with logarithmic singularities (equivalently, the natural map $\on{At}_{(C,D)}(E)\to T_C(-D)$ is equipped with a section $q^\nabla$) then $s$ extends to a first order neighborhood in the universal isomonodromic deformation of $E$ if and only if the map $$\psi\circ q^\nabla: T_C(-D)\to E$$ induces the zero map $$H^1(C, T_C(-D))\to H^1(C, E).$$
\end{lemma}
\begin{proof}
For the first paragraph, see \cite[Proposition 3.3.14]{sernesi2007deformations} for the case where $E$ is a line bundle and $D$ is empty; the general case is identical. For the second, fix $v\in H^1(C, T_C(-D))$, corresponding to a first-order deformation of $(C, D)$. The element $q^\nabla(v)\in H^1(\on{At}_{(C,D)}(E))$ corresponds to, by \cite[Proposition 3.5.7]{landesmanL:geometric-local-systems}, the first-order isomonodromic deformation of $E$ in the direction $v$. Thus by the first paragraph, $\psi\circ q^\nabla(v)=0$; as $v$ was arbitrary, this completes the proof.
\end{proof}

\begin{proposition}
	\label{proposition:globally-generated}
	Fix a unitary representation $\rho: \pi_1(\Sigma_{g,n}) \to \on{GL}_r(\mathbb
	C)$.
	Let
	$(C, D)$ be a general $n$-pointed curve of genus $g \geq 2+2r$.
	Then $\widehat{E}^\rho_0 \otimes \omega_C(D)$ is not only generically globally
	generated, but even globally generated.
\end{proposition}
\begin{proof}
It suffices to consider the case that $\rho$ is irreducible and non-trivial (as
$\omega_C$ is globally generated). Letting $\mathbb V^\rho$ denote the local
system on $C^\circ$ associated to $\rho$, we may assume that $H^1(C, \widehat
E^\rho_0\otimes \omega_C(D))=0$, as this is dual to $H^0(C,
E^{\rho^\vee}_0)=H^0(C^\circ, \mathbb V^{\rho^\vee})$. 

Suppose that the theorem is false, so that for a general $n$-pointed curve
$(C,D)$, $\widehat{E}^\rho_0 \otimes \omega_C(D)$ is not globally generated. Equivalently, there exists $p\in C$ such that $H^1(C,\widehat{E}^\rho_0 \otimes \omega_C(D-p))\neq 0$. Serre-dually, $H^0(C, E_0^{\rho^\vee}(p))$ is non-zero. As $(C,D)$ is general, we may assume that there exists a section $s\in H^0(C, E_0^{\rho^\vee}(p))$ such that for every first-order deformation $(\widetilde C, \widetilde D)$ of $(C,D)$, there exists a first-order deformation $\widetilde p$ of $p$ to $\widetilde C$ so that $s$ survives to $H^0(\widetilde C, \widetilde E_0^{\rho^\vee}(\tilde p)),$ where $\widetilde E_0^{\rho^\vee}$ is the isomonodromic deformation of $E_0^{\rho^\vee}$ to $(\widetilde C, \widetilde D)$. Note that if $p\in D$ but $\tilde p\not\subset \widetilde{D}$, then we may, by passing to a different general curve $(C', D')$, assume that $p\not\in D$. Thus we can and do assume in what follows that if $p\in D$, then $\tilde p\subset \widetilde D$ for all first-order deformations $(\widetilde C, \widetilde D)$ of $(C,D)$.

There are two cases, depending on whether or not $p\in D$; we set up notation to handle them simultaneously. If $p\in D$, we set $D'=D$; otherwise, we set $D'=D+p$. The bundle $\mathscr{O}(p)$ has a natural connection on it with regular singularities at $p$ and trivial monodromy, and with residue $-1$ at $p$. We give $E^{\rho^\vee}_0(p)$ the tensor product connection, which has regular singularities along $D'$.

First order deformations of $(C,D)$ are parameterized by $H^1(C, T_C(-D))$, and first order deformations of $(C, D')$ are parameterized by $H^1(C, T_C(-D'))$. By \autoref{lemma:atiyah-vanishing-h1}, applied to $(C, D')$ and the vector bundle $E^{\rho^\vee}_0(p)$,
a section $s\in H^0(C, E_0^{\rho^\vee}(p))$ extends to the first-order
deformation parameterized by $\eta\in H^1(C, T_C(-D'))$ if and only if the image of $\eta$
under the map $$H^1(C, T_C(-D'))\overset{H^1(\nabla s)}{\longrightarrow}
H^1(E_0^{\rho^\vee}(p))$$ is zero, where $\nabla s: T_C(-D')\to E_0^{\rho^\vee}(p)$
is the map sending a vector field $X$ to $\nabla_Xs$, the derivative of $s$ with
respect to the natural connection on $E_0^{\rho^\vee}(p)$ with regular
singularities along $D'$, described in the previous paragraph. 
Note that $\nabla s$ is non-zero, as if it
were zero, $s$ would be flat, which is impossible as the monodromy
representation $\rho^\vee$ is non-trivial.

Thus it suffices to show that $$\ker(H^1(C, T_C(-D'))\overset{H^1(\nabla s)}{\longrightarrow} H^1(E_0^{\rho^\vee}(p)))$$ does not surject onto $H^1(C, T_C(-D))$ under the natural map induced by the inclusion $T_C(-D')\to T_C(-D)$.

Since $\deg D' \geq \deg D - 1$, it is enough to show that the rank of the map
induced by $\nabla s$ on $H^1$ is at least $2$, or equivalently that the Serre
dual map $$\widehat E_0^\rho\otimes \omega_C(D-p)\to \omega_C^{\otimes 2}(D')$$
induces a map of rank at least $2$ on $H^0$. Setting $U$ to be the kernel of
this map, $\mu:=2g-3$, and $\delta$ to be the rank of the map $$H^0(\widehat
E_0^\rho\otimes \omega_C(D-p))\to H^0(\omega_C^{\otimes 2}(D')),$$ we have by \autoref{lemma:ggg-parabolic-consequence} that $$r\geq \frac{g-\delta}{2},$$ and hence that $\delta\geq g-2r.$ Hence $\delta\geq 2$ as long as $g-2r\geq 2$, which holds by assumption.
\end{proof}

\section{Preliminaries on variations of Hodge structure}
\label{section:vhs}
In this section, we freely use the terminology of Tannakian categories. For
background, we suggest the reader consult
\cite[II]{deligneM:tannakian-categories}.

Let $Y$ be a smooth variety over $\mathbb C$ and $x$ a point of $Y$. 
\begin{definition}[Integral $K$-VHS]
	Let $K$ be a number field. A $K$-variation of Hodge structures on $Y$ is a $K$-local system $\mathbb{V}$ on $Y$ equipped with a decreasing filtration $F^\bullet$ on $W_{K/\mathbb{Q}}\mathbb{V}\otimes \mathscr{O}_Y$, where $W_{K/\mathbb{Q}}$ is the Weil restriction, turning $W_{K/\mathbb{Q}}\mathbb{V}$ into a polarizable $\mathbb{Q}$-variation of Hodge structure so that the natural action of $K$ is via morphisms of variations of Hodge structure. We say that a $K$-variation $\mathbb{V}$ is \emph{integral} if there exists a locally constant sheaf of locally-free $\mathscr{O}_K$-modules $\mathbb{W}$ on $Y$ and an isomorphism $\mathbb{W}\otimes_{\mathscr{O}_K} K\simeq \mathbb{V}$. 
	
	We denote by $\on{VHS}(Y, K)$ the neutral Tannakian category of
	semisimple $K$-variations of mixed Hodge structure (i.e.~direct sums of irreducible pure
	$K$-variations of Hodge structure),
	and by $\on{VHS}(Y, \mathscr{O}_K)$ the full (neutral Tannakian) sub-category consisting of direct sums of integral pure $K$-variations of Hodge structure. Note that $\on{VHS}(Y, \mathscr{O}_K)$ is a $K$-linear category, not an $\mathscr{O}_K$-linear category. We equip this category with the fiber functor sending a local system to its fiber at $x$.
	
	If $\mathbb{V}$ is a $\mathbb{Q}$-VHS, the \emph{generic Mumford-Tate
	group} of $\mathbb{V}$ is the identity component of the Tannakian group
	associated to the full subcategory of $\on{VHS}(Y, \mathbb{Q})$
	generated by $\mathbb{V}$ (see e.g.~\cite[Penultimate paragraph of \S4.1]{moonen2017families}
	for more discussion and a comparison to other definitions of the generic Mumford-Tate group; in particular, if $x$ is very general, there is a canonical isomorphism between the generic Mumford-Tate group and the Mumford-Tate group of $\mathbb{V}_x$).
	
	For any local system $\mathbb{V}$  on a variety  $Y$, with associated monodromy representation $\rho: \pi_1(Y, y)\to \on{GL}(\mathbb{V}_y)$, the \emph{algebraic monodromy group} of $\mathbb{V}$ is the identity component of the Zariski closure of the image of $\rho$.
\end{definition}
\begin{lemma}
	Let $\mathbb{V}$ be an integral $K$-variation of Hodge structure on a smooth variety $Y$. Then the algebraic monodromy group of $\mathbb{V}$ is a normal subgroup of the derived subgroup of the generic Mumford-Tate group of the Weil restriction $W_{K/\mathbb{Q}}\mathbb{V}$.
\end{lemma}
\begin{proof}
This is immediate from \cite[Theorem 1 on p. 10]{AndreFixed}. 
\end{proof}

If $\mathbb{V}$ is a $K$-variation of Hodge structure on $Y$ and $\iota: K\hookrightarrow \mathbb{C}$ is an embedding, then $\mathbb{V}\otimes_{K, \iota} \mathbb{C}$ naturally obtains the structure of a $\mathbb{C}$-variation of Hodge structure.

\begin{definition}[Infinitesimal VHS]
	A  \emph{weak infinitesimal variation of Hodge structure on $Y$ at $x$},
	or weak IVHS, is a finite-dimensional $\mathbb{Z}$-graded complex vector
	space $V^\bullet$ equipped with an action by $T_xY$ by commuting linear
	operators of degree $-1$, i.e.~with a map $\delta^i: T_xY\to
	\on{Hom}(V^i, V^{i-1})$, such that
	$\delta^{i-1}(w_1)\circ\delta^i(w_2)(v)=\delta^{i-1}(w_2)\circ\delta^i(w_1)(v)$
	for all $w_1, w_2\in T_xY, v\in V^i$. A morphism of weak IVHS is a
	graded map of vector spaces commuting with the action of $T_xY$. We
	denote the category of weak IVHS on $Y$ at $x$ by $\on{IVHS}(Y,x)$. 
\end{definition}
\begin{remark}
We call the above \emph{weak} IVHS because the conditions we impose are weaker than the standard conditions on infinitesimal variations of Hodge structure; see for example \cite{carlsonGGH:ivhsI}.
\end{remark}
\begin{proposition}\label{proposition:ivhs-tannakian}
	The category $\on{IVHS}(Y,x)$ is a neutral Tannakian category, and the forgetful functor $\on{IVHS}(Y, x)\to \on{Vect}_\mathbb{C}$ sending $V^\bullet$ to its underlying vector space is a fiber functor. The corresponding Tannakian group is $T_xY\rtimes \mathbb{G}_m$, where $\mathbb{G}_m$ acts on $T_xY$ by inverse scaling.
\end{proposition}

\begin{proof} It suffices to prove that $\on{IVHS}(Y,x)$ is equivalent to the category of representations of $T_xY\rtimes \mathbb{G}_m$, with the equivalence respecting the tensor product and forgetful functor to vector spaces, as the category of representations of any pro-algebraic group is a neutral Tannakian category with fiber functor the forgetful functor and Tannakian group the initial pro-algebraic group \cite[II, Example 1.25]{deligneM:tannakian-categories}.

Given a weak infinitesimal variation of Hodge structures, we obtain an action of the algebraic group $T_xY$ on $V^\bullet$ by exponentiating $\delta$, i.e. for $v \in V^i$ and $w\in T_x Y$ we have
\[ w \cdot v = v + \delta^i(w) ( v) + \frac{ \delta^{i-1} (w) \circ \delta^i(w) (v) }{2!} + \frac{ \delta^{i-2} (w) \circ  \delta^{i-1} (w) \circ \delta^i(w) (v) }{3!} +\dots.\]
(The power series converges since $V^{i-k}=0$ for all $k$ sufficiently large.) The relation $\delta^{i-1}(w_1)\circ\delta^i(w_2)(v)=\delta^{i-1}(w_2)\circ\delta^i(w_1)(v)$ implies that $ w_1 \cdot (w_2 \cdot v) =(w_1+ w_2) \cdot v$ and thus this is an action of the additive group $T_xY$ on $V^\bullet$.

We also obtain an action of $\mathbb G_m$ on $V^\bullet$ by, for $v\in V^i$ and $\lambda \in \mathbb G_m$, taking $\lambda \cdot v= \lambda^i v$. Then for $\lambda \in \mathbb G_m$, $w\in T_x Y$ and $v \in V^\bullet$ it is not hard to check that we have $\lambda \cdot (w\cdot (\lambda^{-1} \cdot v))= (\lambda^{-1} w ) \cdot v$. 
This relation implies that the two actions combine to give an action of $T_xY\rtimes \mathbb{G}_m$ on $V^\bullet$.

Conversely, given a vector space $V$ with an action of $\mathbb G_m$, we obtain a grading by taking $V^i$ to be the subspace on which $\lambda\in \mathbb G_m$ acts with eigenvalue $\lambda^i$ for all $\lambda$, and the derivative at the identity of the action of the additive group $T_x Y$ defines an action of $T_x Y$ by commuting linear operators.


These two constructions are inverse (as can be checked separately on the $T_x Y$ and $\mathbb G_m$ parts, the first part being the standard equivalence between representations of a unipotent algebraic group and nilpotent representations of its Lie algebra, and the second being the standard equivalence between graded vector spaces and representations of $\mathbb G_m$).

They are compatible with tensor product (for the natural notion of tensor product on weak infinitesimal variations of Hodge structures) and, trivially, with the forgetful functor to the underlying vector space. 
\end{proof}

Given a $K$-variation of Hodge structure $\mathbb V$ on $Y$ and an embedding $\iota: K\hookrightarrow\mathbb{C}$, 
we can form the graded vector space $\bigoplus_{i\in \mathbb Z} (F^i \mathbb V_x\otimes_{K, \iota} \mathbb{C}) / (F^{i+1} \mathbb V_x \otimes_{K, \iota} \mathbb{C})$, which admits an action of the tangent space $T_x Y$ by commuting linear maps of degree $-1$ (by Griffiths transversality). This formation is functorial, giving rise to a functor 
$\on{GH}_x$, where $\on{GH}$ stands for for ``graded Hodge,''
\begin{align}
	\label{equation:infinitesimal-functor}
	\on{GH}_x: \on{VHS}(Y, K)\to \on{IVHS}(Y, x)
\end{align}
(which depends on $\iota$, but we suppress this dependence from the notation). 

Composing this functor with the Weil restriction functor $W_{K/\mathbb{Q}}$ induces a homomorphism from $T_xY\rtimes \mathbb{G}_m$ to the generic Mumford-Tate group of $W_{K/\mathbb{Q}}\mathbb{V}$, since $T_xY\rtimes \mathbb{G}_m$ is connected and hence any homomorphism from $T_xY\rtimes \mathbb{G}_m$ to the Tannakian group lands in its identity component.
\begin{lemma}
	\label{lemma:simple-or-reducible}
	Let $K$ be a number field, $Y$ a smooth complex variety, $x$ a point of $Y$, and $\mathbb V$ a pure integral
	$K$-VHS on $Y$. Let $M$ be the monodromy group of $\mathbb
	V$, i.e. the Zariski closure of the monodromy representation  $\pi_1(Y, x) \to
	\on{GL}(\mathbb{V}_x)$, 
	and let $\mathfrak m$ be the Lie algebra of the identity component of $M$.

Suppose $\mathbb V$ has at most $k+1$ nonvanishing Hodge numbers. Then either $\on{GH}_x(\mathbb V)$ 
splits as a direct sum in $\on{IVHS}(Y,x)$ or $\mathfrak m$ acts irreducibly on $\mathbb V$ and has at most $k$ simple factors. 
 
 In particular, if $\mathbb V$ has at most $2$ nonvanishing Hodge numbers, then either $\on{GH}_x(\mathbb V)$ splits or the identity component of $M$ is a simple algebraic group acting irreducibly on $\mathbb V$.
 \end{lemma}
 
 \begin{proof} Since monodromy groups are preserved when a local system is extended to a larger coefficient fields, we may assume $K$ is Galois over $\mathbb Q$.  
 
Let $G$ be the generic Mumford-Tate group of $W_{K/\mathbb{Q}}\mathbb{V}$ 
and $\mathfrak g$ the Lie algebra of $G$. The action of $T_x Y\rtimes \mathbb G_m$ on $\on{gr}^{F^\bullet}(W_{K/\mathbb{Q}}\mathbb{V}_x\otimes \mathbb{C})$ is given by a homomorphism $T_xY\rtimes \mathbb G_m \to G_{\mathbb{C}}$.

Because the category $\on{VHS}(Y, \mathscr{O}_K)$ is semisimple, $G$ is reductive.
Thus if $\mathbb V$ is reducible as a representation of $G_K$ (that is, after passing to a finite \'etale cover, as a $K$-VHS), it splits as a direct sum of two representations of $G$, hence a direct sum of two representations of  $T_x Y\rtimes \mathbb G_m$, so  $\on{GH}_x(\mathbb V)$ splits as a sum of two graded vector spaces with actions of $T_x Y$. Hence we may assume $\mathbb V$ is irreducible as a representation of $G_K$ and thus irreducible as a representation of $\mathfrak g$.

Because $G$ is reductive, $\mathfrak g$ splits as a product of $n$ simple Lie algebras $\mathfrak g_1,\dots, \mathfrak g_n$ times a trivial Lie algebra, and because $\mathbb V$ is irreducible as a representation of $\mathfrak g$, it is a tensor product of $n$ nontrivial irreducible  representations $V_1,\dots, V_n$ of the $n$ simple Lie algebras with a one-dimensional representation of the trivial Lie algebra. Each $V_i$ is a representation of the Lie algebra of $T_x Y\rtimes \mathbb G_m$, and thus admits a $\mathbb C$-grading and an action of $T_x Y$ by commuting linear maps of degree $-1$.

If $T_x Y$ acts trivially on some $V_i$, then $V_i$ splits as a direct sum since every graded vector space of dimension $>1$ splits as a direct sum of graded vector spaces (and one-dimensional representations of simple Lie algebras are trivial), 
so $\mathbb V$ splits as a direct sum of graded vector spaces with actions of $T_x Y$. So we may assume $T_x Y$ acts nontrivially on each $V_i$. It follows that each $V_i$ has vectors of at least two different grades, so the tensor product $\mathbb V$ of the $V_i$ has vectors of at least $n+1$ different grades, and thus $n \leq k$. 

As $M$  is a normal subgroup of $G$, the Lie algebra $\mathfrak m$ of the
identity component of $M$ is an ideal of $G$. Therefore, $\mathfrak m$ is a sum of some of the $\mathfrak g_i$'s, possibly with a trivial algebra. Since monodromy groups of $\mathbb Q$-VHS's are simple, $\mathfrak m$ is a sum of some of the $\mathfrak g_i$s. If some $\mathfrak g_i$ does not appear in this sum, then its adjoint representation corresponds to a $\mathbb Q$-VHS (possibly on a cover of $X$) with trivial monodromy, hence a constant Hodge structure, so the derivative of its period map vanishes, and thus $T_x Y$ acts trivially on this adjoint representation. But this implies that the image of $T_x Y$ in $\mathfrak g_i$ is zero, and thus $T_x Y$ acts trivially on $V_i$, contradicting our assumption that $T_x Y$ acts nontrivially on each $V_i$.

Thus $\mathfrak m$ is the sum of all the $\mathfrak g_i$s and thus acts irreducibly on $\mathbb V$, and, in addition, has $n\leq k$ simple factors.
\end{proof}

\subsection{Lie algebras and weights}\label{subsection:weights}

We discuss some generalities on representations of the Lie algebra of the generic Mumford-Tate group.

For $\mathbb V$ a $\mathbb Q$-VHS on a variety $X$ and $x$ a point of $X$, \autoref{proposition:ivhs-tannakian} gives a homomorphism $T_x Y \rtimes \mathbb G_m \to G$ where $G$ is the generic Mumford-Tate group of $\mathbb V$. Thus we obtain a Lie algebra homomorphism $T_x Y \rtimes \mathbb C \to \mathfrak g$, where $\mathfrak g$ is the Lie algebra of $G$, $\mathbb C$ is the Lie algebra of $\mathbb G_m$, and $\mathbb C$ acts on $T_x Y$ by $ [1, v] =-v$ for $ v\in T_x Y$, where the minus sign appears because $\mathbb{G}_m$ acts on $T_x Y$ by inverse scaling, by \autoref{proposition:ivhs-tannakian}.

For any representation of $\mathfrak g$, we refer to the eigenvalues of $\mathbb C \subseteq T_x Y \rtimes \mathbb C$ on that representation as the weights and the generalized eigenspace of a given eigenvalue as the weight space. For the representation arising from the action of $G$ on $\mathbb V_x$, these weights agree with the Hodge weights of $\mathbb V$, and in particular are integers, but for arbitrary representations they will be complex numbers.

For any representation that factors through a finite covering of $G$, the action of $\mathbb C$ factors through a finite covering of $\mathbb G_m$. This implies that the weights will be rational numbers, and the action of $\mathbb C$ is semisimple so the generalized eigenspaces will be eigenspaces.

The identity $[1,v] =-v$ for $v \in T_x Y$ implies that the elements of $T_x Y$ send elements of weight $w$ to elements of weight $w-1$.


\begin{example}\label{ex:wedge-example}
	Let $\mathbb{V}$ be an integral $K$-variation of Hodge structure on $X$, with Mumford-Tate (isogenous to) $\on{GL}_\nu$ acting through the $\wedge^k: \on{GL}_{\nu}\to \on{GL}_{{\nu \choose k}}$, the $k$-th wedge power of the standard representation, where $1<k<\nu-1$. Let $\on{std}: \mathfrak{gl}_\nu\to \mathfrak{gl}_\nu$ be the standard representation. As in the discussion above, we have a Lie algebra homomorphism $\iota: T_xY\rtimes \mathbb{C}\to \mathfrak{g}=\mathfrak{gl}_\nu$, so that the weights of $\wedge^k\circ \iota(1)$ are the Hodge weights of $\mathbb{V}$.
	
	We may in this case consider $\on{std}\circ \iota(1)$, which acts on $\mathbb{C}^{\nu}$ with weights in $\frac{1}{k}\mathbb{Z}$. By definition, $\wedge^k\circ\iota(1)=\bigwedge^k(\on{std}\circ \iota(1))$. Thus, letting $a_1, \cdots, a_\nu$ be the weights of $\on{std}\circ \iota(1)$, we have that the weights of $\wedge^k\circ\iota(1)$ are precisely the sums $a_{i_1}+ \cdots+ a_{i_k}$, where $1\leq i_1, \cdots, i_k\leq \nu$ are distinct integers.
	
	The following case will be used in the proof \autoref{lemma:wedge-bound}. Suppose every weight of $\wedge^k\circ\iota(1)$ is either $0$ or $1$. Then the weights of $\on{std}\circ \iota(1)$ must be either $$\{\frac{1}{k}, \cdots, \frac{1}{k}, -\frac{k-1}{k}\}\text{ or } \{0, \cdots, 0, 1\},$$ by an elementary combinatorial analysis.

\end{example}

\section{Generic Torelli theorems for unitary local systems}
\label{section:generic-torelli}

Let $\pi: \mathscr{C}\to \mathscr{M}$ be a versal family of $n$-punctured curves
of genus $g$, and let $\pi^\circ: \mathscr{C}^\circ\to \mathscr{M}$ be the
associated punctured versal family, defined as in
\autoref{notation:versal-family-of-covers}. Let $m\in \mathscr{M}$ be a general
point and $C=\mathscr{C}_m, C^\circ=\mathscr{C}^\circ_m, D=C - C^\circ$. Let $\mathbb{U}$ be a unitary local system on $\mathscr{C}^\circ$. The goal of this section is to show that, in many cases, $\mathbb{U}|_{C^\circ}$ can be functorially recovered from the local system $W_1R^1\pi^\circ_*\mathbb{U}$.
In fact, it will be recoverable from the associated weak IVHS at a general point of $m$. We view this as a generic Torelli theorem for curves with a unitary local system, and indeed the proof is closely related to classical proofs of the generic Torelli theorem, as in \cite{harris:IVHS-notes}.

\subsection{Reconstructing unitary bundles}
In this section we reconstruct from $W_1R^1\pi_*^\circ\mathbb{U}$ the vector bundle $\widehat E_0$ corresponding to the parabolic bundle on $(C,D)$ associated to $\mathbb{U}|_{C^\circ}$. In fact we will reconstruct $\widehat E_0$ 
from $\on{GH}_m(W_1R^1\pi_*^\circ \mathbb{U}),$
with $\on{GH}$ as defined in \eqref{equation:infinitesimal-functor}.
In case $D=\emptyset$, this in fact recovers $\mathbb{U}|_C$ by the
Narasimhan-Seshadri correspondence \cite{narasimhan1965stable}. When $D$ is
non-empty, some additional work will be required to recover $\mathbb{U}|_C$.

\begin{theorem}\label{thm:vector-bundle-reconstruction}
		 With notation as above, let $\mathbb{U}$ be a unitary local system on $C^\circ$, with monodromy representation $\rho$ and associated parabolic bundle $E_\star$ on $(C,D)$. Let $r=\rk(\mathbb{U})$, and assume $g\geq 2+2r$. Let 
	 $\psi^\rho_m$ be the map $$H^0(C, \widehat{E}_0\otimes\omega_C(D))\otimes\mathscr{O}_C\overset{\theta^\rho_m\otimes \on{id}}{\longrightarrow} H^1(C, E_0)\otimes H^0(C, \omega_C^{\otimes 2}(D))\otimes \mathscr{O}_C\overset{\on{id}\otimes \on{ev}}{\longrightarrow} H^1(C, E_0)\otimes \omega_C^{\otimes 2}(D)$$
	 obtained as the composition of the map induced by $\theta_m^\rho$, defined in \eqref{equation:theta-map}, with the map induced by the evaluation map $\on{ev}: H^0(C, \omega_C^{\otimes 2}(D))\otimes \mathscr{O}_C\to \omega_C^{\otimes 2}(D).$ Then $\psi^\rho_m$ factors through the evaluation map $$\on{ev}: H^0(C, \widehat{E}_0\otimes\omega_C(D))\otimes\mathscr{O}_C\to \widehat{E}_0\otimes\omega_C(D),$$ and induces an isomorphism $$\widehat{E}_0\otimes\omega_C(D)\overset{\sim}{\to}\on{im}(\psi^\rho_m).$$
\end{theorem}
\begin{proof}

It suffices to show that 
\begin{align}
	\label{equation:kernel-equality}
\ker(\psi_m^\rho)=\ker(\on{ev}: H^0(C, \widehat{E}_0\otimes\omega_C(D))\otimes\mathscr{O}_C\to \widehat{E}_0\otimes\omega_C(D)),
\end{align}
as $\widehat E_0\otimes \omega_C(D)$ is globally generated by \autoref{proposition:globally-generated}, using that $m$ is a general point of $\mathscr{M}$.

We first observe that $H^1(C, E_0)$ is, by Serre duality, dual to $H^0(C,  E_0^\vee\otimes \omega_C)$. Note that, as $m$ is general, $E_0^\vee\otimes \omega_C$ is globally generated by \autoref{proposition:globally-generated}, as, setting $F_\star$ to be the parabolic dual to $E_\star$, we have $E_0^\vee\otimes \omega_C=\widehat F_0\otimes \omega_C(D)$.

Let $s$ be a local section to $H^0(C, \widehat E_0\otimes
	\omega_C(D))\otimes \mathscr{O}_C$. Now $s$ is in the kernel of
	$\psi^\rho_m$ if and only if, for all $t\in H^1(C, E_0)^\vee=H^0(C,
	E_0^\vee\otimes \omega_C)$, we have that $t(\psi^\rho_m(s))=0$ as a
	local section to $\omega_C^{\otimes 2}(D)$. Note that
	$t(\psi^\rho_m(s)) = B^\rho_m(s, t)$, for $B^\rho_m$ as defined in
	\eqref{key-bilinear-form}.
	Now, $B^\rho_m(s, t)$
	vanishes for all $t$ if and only if $\on{ev}(s)=0$, as $E_0^\vee\otimes \omega_C$ is (generically) globally generated and $B^\rho_m$ induces a perfect pairing between the generic fibers of $\widehat E_0\otimes \omega_C(D)$ and $E_0^\vee\otimes \omega_C$.
	This implies \eqref{equation:kernel-equality}.	
\end{proof}
\subsection{Functoriality}
The upshot of \autoref{thm:vector-bundle-reconstruction} is that $\widehat{E}_0\otimes \omega_C(D)$ (and hence $\widehat E_0$ itself) may be constructed from $\on{GH}_m(W_1R^1\pi_*\mathbb{U})$ for generic $m$. Note that this construction is functorial: given $\mathbb{U}, \mathbb{U}'$ unitary local systems on $\mathscr{C}^\circ$, and given a map $\on{GH}_m(W_1R^1\pi^\circ_*\mathbb{U})\to \on{GH}_m(W_1R^1\pi_*^\circ \mathbb{U}')$, one obtains a map $\widehat E_0\to \widehat E_0'$, where $E_\star'$ is the parabolic bundle on $C$ corresponding to $\mathbb{U}'|_{C^\circ}$. The goal of this section is to show that in many cases this map is necessarily flat for the natural unitary connections on $\widehat E_0, \widehat E_0'$. This is automatic from the Narasimhan-Seshadri correspondence if $D=\emptyset$, but not in general.

\begin{proposition}\label{proposition:functoriality}
	With notation as in \autoref{notation:versal-family-of-covers}, let
	$m\in \mathscr{M}$ be general, set $C=\mathscr{C}_m,
	C^\circ=\mathscr{C}^\circ_m, D=C - C^\circ$.  Let $\mathbb{U}, \mathbb{U}'$ be unitary local systems on $\mathscr{C}^\circ$, such that $\mathbb{U}|_{C^\circ}, \mathbb{U}'|_{C^\circ}$ are irreducible, of dimension $r, r'$. Assume that $$\on{GH}_m(W_1R^1\pi^\circ_*\mathbb{U}), \on{GH}_m(W_1R^1\pi^\circ_*\mathbb{U}')$$ are isomorphic to one another. Suppose that $g\geq \max(2+2r, 2+2r')$ and either
	\begin{enumerate} \item $D =\emptyset$, or 
	\item $g> r r'$.
	\end{enumerate} 
	Then the natural map \begin{align}\label{equation:hom-isomorphism} \on{Hom}_{\on{LocSys}}(\mathbb{U}|_{C^\circ}, \mathbb{U}'|_{C^\circ})\to \on{Hom}_{\on{IVHS}(\mathscr{M},m)}(\on{GH}_m(W_1R^1\pi^\circ_*\mathbb{U}), \on{GH}_m(W_1R^1\pi^\circ_*\mathbb{U}'))\end{align}
 is a bijection. In particular, $\mathbb{U}|_{C^\circ}$ is isomorphic to $\mathbb{U}'|_{C^\circ}$.
\end{proposition}
\begin{proof}
Let $\rho, \rho'$ be the monodromy representations of $\mathbb{U}, \mathbb{U}'$, and let $E^\rho_\star, E^{\rho'}_\star$ be the parabolic bundles associated to these local systems. By \autoref{thm:vector-bundle-reconstruction}, there exists an isomorphism $\widehat E^\rho_0\simeq \widehat E^{\rho'}_0$, and hence $r=r'$. 

We now prove \eqref{equation:hom-isomorphism} is
 injective. Indeed, a map $\mathbb{U}|_{C^\circ}\to  \mathbb{U}'|_{C^\circ}$ is determined by the induced map $\widehat{E}_0^\rho\to \widehat{E}_0^{\rho'}$, which is in turn determined by the induced map $$H^0(C, \widehat{E}_0^\rho\otimes \omega_C(D))\to H^0(C, \widehat{E}_0^{\rho'}\otimes \omega_C(D)),$$ as the vector bundles in question are globally generated by \autoref{proposition:globally-generated}.

To prove surjectivity of \eqref{equation:hom-isomorphism}, consider a map $\psi:
\on{GH}_m(W_1R^1\pi^\circ_*\mathbb{U})\to \on{GH}_m(W_1R^1\pi^\circ_*\mathbb{U}')$. By \autoref{thm:vector-bundle-reconstruction}, it induces a map $\widetilde\psi: \widehat{E}_0^\rho\to \widehat{E}_0^{\rho'}.$ We must show this map is flat. Equivalently, we wish to show that the map $$\nabla\widetilde\psi: T_C(-D)\to \on{Hom}(\widehat E^\rho_0, \widehat E^{\rho'}_0)$$ sending a vector field $X$ to $\nabla_X\widetilde\psi$ is identically zero. Here $\nabla_X$ is differentiation along $X$, using the connection $\nabla$ on $\on{Hom}(\widehat E^\rho_0, \widehat E^{\rho'}_0)$. When $D=\emptyset$ this is immediate by the functoriality of the Narasimhan-Seshadri correspondence, so we need only consider the case $D\neq \emptyset$.

By \autoref{lemma:atiyah-vanishing-h1}, we know that $\nabla\widetilde \psi$ induces the zero map $$H^1(C, T_C(-D))\to H^1(\on{Hom}(\widehat E^\rho_0, \widehat E^{\rho'}_0)),$$ as $m$ is general and hence $\psi$ extends to a first-order neighborhood of $m$. Serre-dually, the map $$\on{Hom}(\widehat E^{\rho'}_0, \widehat E^{\rho}_0)\otimes \omega_C\to \omega_C^{\otimes 2}(D)$$ obtained by dualizing $\nabla\widetilde \psi$ and tensoring with $\omega_C$ induces zero on $H^0$. Thus if $\nabla\widetilde \psi$ is non-zero, the bundle $\on{Hom}(\widehat E^{\rho'}_0, \widehat E^{\rho}_0)\otimes \omega_C$ is not generically globally generated.

Now as $g>rr'=r^2$, $E^{(\rho')^\vee}_0$ and $E^{\rho^\vee}_0$ are semistable,
by \cite[Corollary 6.1.2]{landesmanL:geometric-local-systems}; hence by \autoref{lemma:dual-hat}, the same is true for $\widehat E^{\rho'}_0, \widehat E^{\rho}_0$. As these bundles
are isomorphic by the first paragraph above, we have that $\on{Hom}(\widehat E^{\rho'}_0, \widehat E^{\rho}_0)\otimes \omega_C$ is semistable of slope $2g-2$ and rank $rr'=r^2$. As
$g> rr'$ it is thus generically globally generated by 
\cite[Proposition 6.3.1(b)]{landesmanL:geometric-local-systems}, whence the proof is complete.
\end{proof}
The following will not be used in what follows, but we felt it might be of independent interest.

\begin{construction}
	With notation as in \autoref{notation:versal-family-of-covers}, let
	$\mathbb U$ be a unitary local system on $\mathscr C^\circ$ of rank $r$.
	Let $g> r^2$. Let $W_1R^1\pi_*^\circ \mathbb{U}$
	be the above defined $\mathbb{C}$-VHS 
	and let  $m\in \mathscr{M}$ be general, with $C^\circ=\mathscr{C}^\circ_m,
	C=\mathscr{C}_m$. 
	We next sketch how to directly reconstruct the connection on
	$\widehat E_0$, where $E_\star$ is the parabolic bundle associated to
	$\mathbb{U}|_{C^\circ}$.

	We may recover the connection from the restriction of $W_1R^1\pi_*^\circ \mathbb{U}$ to a small neighborhood of $m$, though we do not know how to do so from $\on{GH}_m(W_1R^1\pi_*^\circ \mathbb{U})$. 

Recall that the data of a logarithmic connection on a bundle $F$ is the same as an $\mathscr{O}_C$-linear splitting $s$ of the natural map $$\on{At}_{(C,D)}(F)\to T_C(-D)$$ (see e.g.~\cite[Proposition 3.1.6]{landesmanL:geometric-local-systems}). The induced map $$H^1(C, T_C(-D))\to H^1(C, \on{At}_{(C,D)}(F))$$ may be interpreted as the map sending a first-order deformation $(\widetilde C, \widetilde D)$ of $(C,D)$ to the triple $(\widetilde C, \widetilde D, \widetilde E)$, where $\widetilde E$ is the isomonodromic deformation of $E$ \cite[Proposition 3.5.7]{landesmanL:geometric-local-systems}. This latter map may be recovered (taking $F=\widehat E_0$ from $W_1R^1\pi_*^\circ \mathbb{U}$ if $g\geq r^2$, by performing the construction of \autoref{thm:vector-bundle-reconstruction} in families. Serre-dually, we may recover the equivalent data of the Serre-dual map $$H^0(C, (\on{At}_{(C,D)}(\widehat E_0))^\vee\otimes \omega_C)\to H^0(C, \omega_C^{\otimes 2}(D)).$$ Now $\omega_C^{\otimes 2}(D)$ is globally generated, and $(\on{At}_{(C,D)}(\widehat E_0))^\vee\otimes \omega_C$  is generically globally generated by similar arguments to those in the proof of \autoref{proposition:functoriality}. 

As a map of generically globally generated vector bundles is determined by the induced map on global sections, the argument is complete.
\end{construction}

\subsection{Simplicity of the monodromy representation}

We now prove that the monodromy representation is simple.
\begin{theorem}
	\label{theorem:simplicity-of-VHS}
With notation as in \autoref{notation:versal-family-of-covers}, let $\rho: H\to \on{GL}_r(\mathbb{C})$ be an irreducible $H$-representation, and let $\mathbb{V}$ be the
corresponding local system on $\mathscr{C}^\circ$. 
Suppose that either $n=0$ and $g>2r+1$, or $n$ is arbitrary and $g> \on{max}(2r+1, r^2)$.
Then the variation of Hodge structure $W_1R^1\pi^\circ_*\mathbb{V}$ has simple
monodromy group, acting irreducibly.
\end{theorem}
\begin{proof}Note first that $\mathbb{V}$ is an integral variation of Hodge structure because every representation of a finite group is defined over the ring of integers over some number field and is polarizable because representations of finite groups are unitary. Thus $W_1R^1\pi_*^\circ\mathbb{V}$ is an integral variation of Hodge structures because these are stable under derived pushforward and passing to subspaces in the weight filtration. (See e.g. \cite[Theorem 4.1.1]{landesmanL:canonical-representations} for a discussion without the integrality condition.)

Now suppose that the conclusion of the theorem is false, i.e.~that either the monodromy group of $W_1R^1\pi^\circ_*\mathbb{V}$ is not simple, or it does not act irreducibly. Then by \autoref{lemma:simple-or-reducible}, $\on{GH}_m(W_1R^1\pi_*^\circ\mathbb{V})$ splits as a direct sum for general $m\in \mathscr{M}$. Thus we have $$\dim \on{Hom}_{\on{IVHS}(\mathscr{M}, m)}(\on{GH}_m(W_1R^1\pi_*^\circ\mathbb{V}), \on{GH}_m(W_1R^1\pi_*^\circ\mathbb{V}))\geq 2.$$ But taking $\mathbb{U}=\mathbb{U}'=\mathbb{V}$ in \autoref{proposition:functoriality}, 
we have $$\dim \on{Hom}_H(\rho, \rho)\geq 2.$$ But this contradicts the irreducibility of $\rho$, by Schur's lemma.
\end{proof}
\begin{remark}
We could have argued using \autoref{thm:vector-bundle-reconstruction}, instead
of \autoref{proposition:functoriality}, as follows. The splitting of
$\on{GH}_m(W_1R^1\pi_*^\circ\mathbb{V})$ implies, by the construction of
\autoref{thm:vector-bundle-reconstruction}, that $\widehat E_0^\rho$ itself
splits as a direct sum. But forthcoming work of Ramirez-Cote \cite{ramirez-cote},
 following ideas of \cite{landesmanL:geometric-local-systems}, shows that $\widehat E_0^\rho$ is stable for irreducible $\rho$ when $g\geq r^2$, contradicting this splitting.
\end{remark}

\begin{corollary}\label{corollary:local-systems-isomorphic}
	With notation as in \autoref{notation:versal-family-of-covers}, let $\rho_i: H\to \on{GL}_{r_i}(\mathbb{C}), i=1,2$ be irreducible representations, and let $\mathbb{V}_1, \mathbb{V}_2$ be the
corresponding local systems on $\mathscr{C}^\circ$. 
Suppose $g\geq \on{max}(2+2r_1, 2+2r_2)$ and either $D=\emptyset$ or $g> r_1 r_2$. If the local systems $W_1R^1\pi^\circ_*\mathbb{V}_1, W_1R^1\pi^\circ_*\mathbb{V}_2$ are isomorphic, then $\rho_1$ and $\rho_2$ are conjugate.
\end{corollary}
\begin{proof}
By \autoref{theorem:simplicity-of-VHS}, 	the local systems $W_1R^1\pi^\circ_*\mathbb{V}_1, W_1R^1\pi^\circ_*\mathbb{V}_2$ are irreducible. Hence, by the theorem of the fixed part applied to the tensor product  $(W_1R^1\pi^\circ_*\mathbb{V}_1)^\vee \otimes W_1R^1\pi^\circ_*\mathbb{V}_2$, any isomorphism between them is necessarily an isomorphism of $\mathbb{C}$-VHS (up to a shift of the weight and Hodge filtrations, but these shifts must vanish by consideration of the weight and Hodge numbers of both sides), and in particular induces an isomorphism $\on{GH}_m(W_1R^1\pi^\circ_*\mathbb{V}_1)\simeq \on{GH}_m(W_1R^1\pi^\circ_*\mathbb{V}_2)$ for any $m\in \mathbb{M}$. Now the result follows by \autoref{proposition:functoriality}.
\end{proof}

\section{Proofs of \autoref{theorem:main-thm-1} and \autoref{theorem:main-thm-2}: big monodromy for $g$ large}
\label{section:main-proofs}

In this section, we prove our main results, \autoref{theorem:main-thm-1} and
\autoref{theorem:main-thm-2}.
We next describe the possibilities for the connected monodromy groups of the variations of Hodge structure appearing in \autoref{theorem:simplicity-of-VHS} by applying a result
of Deligne (implicit in his classification of Shimura varieties of Abelian type \cite[1.3.6]{deligne1979varietes}, and explicit in \cite[Theorem 0.5.1(b)]{zarhin1984weights}), and proceed to rule out many of these possibilities through several
lemmas. We conclude the proofs in \autoref{subsection:main-proofs}.

\subsection{Describing the possibilities for the monodromy representation}
\label{subsection:monodromy-possibilities}
First, with notation as in \autoref{theorem:simplicity-of-VHS}, the Hodge filtration of $W_1 R^1 \pi_*^\circ \mathbb V$ has only two parts,
and the monodromy group of $W_1 R^1 \pi_*^\circ \mathbb V$ 
is simple by
\autoref{theorem:simplicity-of-VHS}. As it is a normal subgroup of the generic Mumford-Tate group of $W_1 R^1 \pi_*^\circ \mathbb V$, by \cite[Theorem 1 on p.~10]{AndreFixed}, it follows that it must be a simple factor of this group.
It follows from 
\cite[Theorem 0.5.1(b)]{zarhin1984weights} (which
Zarhin attributes to Deligne)
that the monodromy group
is isogenous to $\on{SL}_\nu, \on{Sp}_{2 \nu}, \on{SO}_\nu$
for some $\nu$, acting in one of the following ways:
\begin{enumerate}
	\item the standard representation,
	\item the trivial representation,
	\item the spin or half-spin representation of the spin cover of the group $\on{SO}_\nu$
	\item a wedge power of the standard representation of the group
		$\on{SL}_\nu$.
\end{enumerate}
We first rule out the trivial representation.
\begin{lemma}
	\label{lemma:no-trivial-rep}
	With notation as in \autoref{notation:versal-family-of-covers}, let $\rho: H\to \on{GL}_r(\mathbb{C})$ be an irreducible $H$-representation, and let $\mathbb{V}$ be the
corresponding local system on $\mathscr{C}^\circ$. If $g > 2$, the connected monodromy group of $W_1R^1\pi^\circ_*\mathbb{V}$ is nontrivial.
\end{lemma}
\begin{proof}
	We are free to pass to finite \'etale covers of the base
$\mathscr M$ in our setup. Therefore, we may assume the Zariski closure of
monodromy is already connected. We also know the monodromy representation is irreducible 
by \autoref{theorem:simplicity-of-VHS}.
Further, $$\dim W_1R^1\pi_*^\circ\mathbb{V}_m\geq (2g-2)\on{rk}(\mathbb{V})>1,$$ since we are assuming $g \geq 2$,
so we conclude that the
monodromy group is not acting via the trivial representation.
\end{proof}

\subsection{Rank estimates}

\begin{lemma}\label{lemma:rank-recovery-1}
		 With notation as in \autoref{notation:versal-family-of-covers}, let $\mathbb{U}$ be a unitary local
		 system on $\mathscr{C}^\circ$ and $m\in \mathscr{M}$ a general point. Let  $\rho$ be the monodromy representation of $\mathbb{U}|_{C^\circ}$ and
		 let $E_\star$ be the associated parabolic bundle on $(C,D)$. Let
		 $r=\rk(\mathbb{U})$, and assume $g\geq 2+2r$. For $p$ in $C$,
		 consider the one-dimensional space of 
$H^0(C, \omega_C^{\otimes 2}(D))^\vee = T_m\mathscr{M}$
		 spanned by the functional sending a global
		 section to its value at $p$, and then composing with a linear map
		 $\omega_C^{\otimes 2}(D)|_p \to \mathbb C$. 

		 For $p$ a general point of $C$ and $\alpha_p$ an element of the associated one-dimensional subspace of $ T_m\mathscr{M}$, the rank of $dP_m^\rho(\alpha_p) \in \on{Hom}(H^0(C, \widehat{E}_0\otimes\omega_C(D)), H^1(C, E_0))$ is equal to $r$.
\end{lemma}
Here $dP_m^\rho$ is defined as in \autoref{section:parabolic-review}, immediately above \autoref{proposition:derivative-of-period-map}.

\begin{proof} \autoref{thm:vector-bundle-reconstruction} shows that
	$\widehat{E}_0\otimes\omega_C(D)$, a sheaf of rank $r$, is the image of
	the map $\psi^\rho_m$, defined as in the statement of that theorem. For any map $f: V \to W$ vector bundles, the rank of the
	image of $f$ is equal to the rank of the fiber $f_q: V_q \to W_q$ of $f$
	at a general point $q$. So it suffices to check that the rank of the
	fiber of $\psi^\rho_m$ at a general point $p$ in $C$ is equal to the rank of $dP_m(\alpha_p)$.

The fiber of the evaluation map $\on{ev}: H^0(C, \omega_C^{\otimes 2}(D))\otimes \mathscr{O}_C\to \omega_C^{\otimes 2}(D)$ at $p$ is, up to scalars, the linear form $\alpha_p \in (H^0(C, \omega_C^{\otimes 2}(D))\otimes \mathscr{O}_C)^\vee$.

Since $\psi^\rho_m$ is the composition of $\theta_m^\rho$ with $ \id \otimes \on{ev}$,
the fiber of $\psi^\rho_m$ is the composition of $\theta_m^\rho$ with
$\id\otimes \alpha_p$. Previously, in \eqref{equation:theta-map},
we defined $\theta_m^\rho$ as the adjoint of $dP_m^\rho
\colon H^0(C, \omega_C^{\otimes 2}(D))^\vee \to \on{Hom}(H^0(C,
\widehat{E}_0\otimes\omega_C(D)), H^1(C, E_0))$, namely by $$\theta^\rho_m: H^0(C,
\widehat{E}_0^\rho\otimes \omega_C(D))\to H^1(C, E_0)\otimes H^0(C,
\omega_C^{\otimes 2}(D)).$$ The defining property of this adjoint is that
composing with $\id \otimes \alpha$ for a linear form $\alpha \in H^0(C,
\omega_C^{\otimes 2}(D))^\vee$ gives a map $H^0(C, \widehat{E}_0^\rho\otimes
\omega_C(D))\to H^1(C, E_0)$ defined by $ d P_m^\rho(\alpha)$, so the
composition of $\theta_m^\rho$ with $\id \otimes \alpha_p$ is $dP_m^\rho(\alpha_p)$, as desired.  \end{proof}

\begin{remark}
The subspaces of $H^0(C, \omega_C^{\otimes 2}(D))^\vee=H^1(C, T_C(-D))$ spanned by the $\alpha_p$ of \autoref{lemma:rank-recovery-1} are known as \emph{Schiffer variations}, see e.g.~\cite[1.2.4]{collino-pirola}. There is some history of using (variants of) Schiffer variations for Torelli-style results, e.g.~in \cite{voisin2022schiffer}.
\end{remark}

\begin{corollary}\label{corollary:rank-recovery-2}
With notation as above, let $\mathbb{U}$ be a unitary local system on $C^\circ$, with monodromy representation $\rho$ and associated parabolic bundle $E_\star$ on $(C,D)$. Let $r=\rk(\mathbb{U})$, and assume $g\geq 2+2r$.  We have 
$$ r \geq \inf \{ \operatorname{rank} ( dP_m^\rho (\alpha)) \mid \alpha \in T_m\mathscr{M} , dP_m^\rho (\alpha) \neq 0 \} .$$
\end{corollary}
\begin{proof} This follows from \autoref{lemma:rank-recovery-1} since any member of a set is at least its minimal value, and $r$ is a member of this set because $r =\operatorname{rank} ( dP_m^\rho (\alpha_p))$ with $dP_m^\rho (\alpha_p)\neq 0$, as $r>0$. \end{proof}

\begin{lemma}\label{lemma:rr-lower-bound} For $\rho$ a unitary representation of $\pi_1(C- D)$ of rank $r$, we have
$$ \dim H^1(C, E_0^\rho) \geq (g-1) r $$
$$ \dim H^0 (C, \widehat{E}_0^\rho \otimes \omega_C (D)) \geq (g-1) r $$
\end{lemma}

\begin{proof} Since $E^\rho_\star$ has parabolic degree $0$, the degree of
	$E_0^\rho$ is  $\leq 0$ and so by Riemann-Roch, $H^1(C, E^\rho_0) \geq
	(g-1)r$.
	By Serre duality, \cite[Proposition
	2.6.6]{landesmanL:geometric-local-systems},
	$\dim H^0(C, \widehat{E}_0^\rho \otimes \omega_C (D)) = \dim H^1(C,
	\widehat{E}_0^{\rho^\vee})$. The latter is $\geq (g-1)r$ by the first part.
\end{proof}

Combining \autoref{corollary:rank-recovery-2} and \autoref{lemma:rr-lower-bound}, we obtain an inequality expressed only in terms of the weak IVHS of $W_1R^1\pi^\circ_*\mathbb{V}$, i.e.~$$\dim H^1(C, E_0^\rho)\geq (g-1)\inf \{ \operatorname{rank} ( dP_m^\rho (\alpha)) \mid \alpha \in T_m\mathscr{M} , dP_m^\rho (\alpha) \neq 0 \},$$
$$\dim H^0 (C, \widehat{E}_0^\rho \otimes \omega_C (D))\geq (g-1)\inf \{ \operatorname{rank} ( dP_m^\rho (\alpha)) \mid \alpha \in T_m\mathscr{M} , dP_m^\rho (\alpha) \neq 0 \}.$$

\subsection{Ruling out non-standard minuscule representations}

In this section we show that if the monodromy group of $W_1R^1\pi_*^\circ
\mathbb{V}$ is $\sl_\nu$ or $\so_\nu$, it must act via the standard representation. We will use the notation discussed in \autoref{subsection:weights}.

\begin{lemma}
	\label{lemma:wedge-bound}
	With notation as in \autoref{notation:versal-family-of-covers}, let $\rho: H\to \on{GL}_r(\mathbb{C})$ be an irreducible $H$-representation, and let $\mathbb{V}$ be the
corresponding local system on $\mathscr{C}^\circ$. Suppose $g\geq 2r+2$.
	If the connected monodromy group of $W_1R^1\pi^\circ_*\mathbb{V}$ is isogenous to $\on{SL}_\nu$ acting via $\wedge^k$ of the
	standard representation, then either $k = 1$ or $k = \nu-1$.
\end{lemma}
\begin{proof}  The identity component of the normalizer of the monodromy group
$\on{SL}_\nu$ is $\on{GL}_\nu/\mu_{\gcd(k,\nu)}$, with $\mu_{\gcd(k,\nu)}$ embedded as scalar matrices, so its Lie algebra is $\mathfrak{gl}_\nu$; in particular we are provided with a natural Lie algebra homomorphism $T_xY\rtimes \mathbb{C}\to \mathfrak{gl}_\nu$ as in \autoref{subsection:weights}. Let $(a_1,\ldots, a_\nu)$ be the weights of the standard representation of $\mathfrak{gl}_\nu$.
	
	We first determine the possible values for the $a_i$ (see also \autoref{ex:wedge-example}). Since the $k$th wedge power of the standard representation has only weights $0$ and $1$,
	as the Hodge structure 
	$W_1 R^1 \pi_*^\circ \mathbb V$ has only two parts, we must have that
	$\sum_{i \in I} a_i \in \{0,1\}$, for any subset $I\subset \{1, \cdots, \nu\}$ of size $k$.
	By \autoref{lemma:no-trivial-rep}, $k \neq 0, \nu$. Thus \autoref{lemma:eca-1} below shows that the only possibilities for $(a_1, \ldots,
	a_\nu)$, up to reordering, are
	\begin{align*}
		a_1 = 1,\, &a_2 = \cdots = a_\nu = 0		\\
		a_1 = \frac{1-k}{k},\, &a_2 = \cdots = a_\nu = 1/k.
	\end{align*} 
	
	We now consider how elements of $T_m \mathscr{M}$ act on the standard representation of $\mathfrak{gl}_{\nu}$. In the first case, any element of $T_m \mathscr{M}$ must send the weight $1$ space to the weight $0$ space and the weight $0$ space to zero. 
	Since the weight $1$ space is one-dimensional, any element of $T_m \mathscr{M}$ must be zero or nilpotent of rank $1$. In the second case, the element sends the weight $1/k$ space to the weight $\frac{1-k}{k}$ space and the weight  $\frac{1-k}{k}$ space to zero, and again must be zero or nilpotent of rank $1$.
	
The action of a nilpotent element of rank $1$ in $\mathfrak{gl}_\nu$ on the representation $\wedge^k$ has rank $\binom{\nu-2}{k-1}$ by \autoref{lemma:lin-alg-1} below. So the minimum nonzero rank of an element of $T_m \mathscr{M}$ acting on $\mathbb V$ is $\binom{\nu-2}{k-1}$. It follows from \autoref{corollary:rank-recovery-2} that $r \geq \binom{\nu-2}{k-1}$.

On the other hand, $ \dim H^1(C, E_0^\rho)$ and
$ \dim H^0 (C, \widehat{E}_0^\rho) \otimes \omega_C (D) $ are either
respectively equal to $\binom{\nu-1}{k-1}$ and $\binom{\nu-1}{k}$  or
respectively equal to $\binom{\nu-1}{k}$ and $\binom{\nu-1} {k-1}$. By
\autoref{lemma:rr-lower-bound}, these are both at least $(g-1)r$, which gives $$
\binom{\nu-1}{k-1} \geq (g-1) r \geq (g-1)  \binom{\nu-2}{k-1} \textrm{ and }
\binom{\nu-1}{k} \geq (g-1) r \geq (g-1)  \binom{\nu-2}{k-1}.$$  Dividing both
sides by $\binom{\nu-2}{k-1} $ we obtain $\frac{\nu-1}{\nu-k} \geq (g-1)$ and
$\frac{\nu-1}{k} \geq (g-1)$ so  $$1 = \frac{k}{\nu} + \frac{\nu-k}{\nu} <
\frac{k}{\nu-1}+\frac{\nu-k}{\nu-1}  \leq \frac{1}{g-1} + \frac{1}{g-1} $$ which
implies $g-1 < 2$, 
contradicting the assumption that $g \geq 2r + 2 \geq 4$.
\end{proof}

\begin{lemma}\label{lemma:eca-1} Let $\nu$ and $k$ be natural numbers with $1< k
	<\nu$. Let $a_1,\dots, a_\nu$ be a tuple of complex numbers such that each sum of exactly $k$ of the $a_1,\dots,a_\nu$ is either $0$ or $1$, with both possibilities occurring. Then up to reordering we have either 
	\begin{align*}	a_1 = 1,\, &a_2 = \cdots = a_\nu = 0\hspace{10pt}	\textrm{ or }	\\
		a_1 = \frac{1-k}{k},\, &a_2 = \cdots = a_\nu = 1/k.
	\end{align*}\end{lemma}
	
	\begin{proof}  We first observe that there can be at most two
	distinct values appearing in $\{a_1, \ldots, a_\nu\}$, as otherwise one
	could find size $k$ subsets summing to three different values.
	Further, if $a_1$ is distinct from $a_2$, then, after possibly switching
	$a_1$ and $a_2$, all other values of $a_i$
	must agree with $a_2$, using that $1 < k < \nu-1$, or else we could again
	obtain three distinct sums from subsets of size $k$.
	Hence, we must have $a_2 = \cdots = a_\nu$, and then a subset of size $k$ drawn from
	$\{a_2,\ldots, a_\nu\}$ must sum to $0$ or $1$, yielding the two
	possibilities claimed above.\end{proof}
	
	\begin{lemma}\label{lemma:lin-alg-1} Let $\nu\geq 2$ and $k$ be natural numbers and Let $N \in \mathfrak{gl}_{\nu}$ be a nilpotent element of rank $1$. Then the action of $N$ on the representation $\wedge^k$ of $\mathfrak{gl}_\nu$ is by a matrix of rank $\binom{\nu-2}{k-1}$. \end{lemma}
	
	\begin{proof}  We may choose a basis $e_1,\dots, e_\nu$ where the element sends $e_1$ to $e_2$ and each other basis vector to $0$, and then the image is generated by wedges of $e_2$ with $k-1$ of the remaining $\nu-2$ basis vectors $e_3,\dots, e_\nu$. \end{proof}

\begin{lemma}
	\label{lemma:spin-bound}
With notation as in \autoref{notation:versal-family-of-covers}, let $\rho: H\to \on{GL}_r(\mathbb{C})$ be an irreducible $H$-representation, and let $\mathbb{V}$ be the
corresponding local system on $\mathscr{C}^\circ$. Suppose $g\geq 2r+2$. The connected monodromy group of $W_1R^1\pi^\circ_*\mathbb{V}$ cannot be isogenous to $\on{Spin}_\nu$ acting via a spin or
half-spin representation
unless the monodromy group is also a classical group acting via its standard representation. 
\end{lemma}

Note that (as discussed in the first paragraph of the proof of \autoref{lemma:spin-bound}) the group $\on{Spin}_\nu$ acting via the spin representation is isogenous to a classical group via an isogeny sending the spin representation to the standard representation, if and only if $\nu=3,4,5,6,8$.

\begin{proof}[Proof of \autoref{lemma:spin-bound}] We first observe that for $\nu=3,4,5,6$, the spin group is respectively $\on{SL}_2, \on{SL}_2\times \on{SL}_2, \on{Sp}_4, \on{SL}_4$ and the spin representation in each case is the standard representation (or, in $\nu=2$, the standard  representation of one of the two factors). For $\nu=8$, the spin group is not a classical group but each half-spin representation factors through the standard representation of a different quotient of $\on{Spin}_8$ isomorphic to $\on{SO}_8$. (The existence of these three different $\on{SO}_8$ quotients is known as triality). So we may assume that $\nu>8$ or $\nu=7$.

	Now, the identity component of the normalizer of the
	monodromy group is $\on{GSpin}_\nu$, with Lie algebra $\mathfrak {go}_\nu$.
	The weights of the standard representation of $\mathfrak {go}_\nu$ come in
	opposite pairs $b+a_1, b-a_1,\ldots, b+a_k, b-a_k$ if $\nu = 2k$ is even, while
	if $\nu = 2k + 1$ is odd, the weights have the form  $b+a_1, b-a_1,\ldots, b+a_k,
	b-a_k, b$. In either case we may assume all the values $a_i$ are nonnegative.

	First, we need to recall properties of the spin representation, as
	defined in \cite[p. 119-122]{chevalley:the-algebraic-theory-of-spinors}.
	In the case $\nu = 2k$ is even, there are two irreducible half-spin
	representation of dimension $2^{k-1}$ and if $\nu = 2k + 1$ is odd, then there
	is an irreducible spin representation of dimension $2^k$.
	
	The weights of the spin representation are then obtained as the $2^k$
	sums $\frac{\pm a_1 \pm a_2 \pm \cdots \pm a_k}{2} + \frac{kb}{2}$.
	One half-spin representation has weights which are sums as above with an
	odd number of minus signs and the other consists of weights which are sums as
	above with an even number of minus signs. Applying \autoref{lemma:eca-2}
	below, in case (1) if $\nu=2k+1$ is odd (where we have $\nu \geq 7$ so
	that $k\geq 3$) and in case (2) if $\nu=2k$ is even (where we have
	$\nu>8$ so $k> 4$), we can assume $a_1=1, a_2,\dots, a_k=0$.
	
It follows that the action of a generator of the Lie algebra $\mathbb C$ on the standard representation of $\mathfrak{so}_\nu$ has a one-dimensional $1+b$-eigenspace, a one-dimensional $b-1$-eigenspace, and a $\nu-2$-dimensional $b$-eigenspace. Applying \autoref{lemma:lin-alg-2} below, where $w$ is obtained by subtracting the scalar $b$ from the generator $1$ of the Lie algebra $\mathbb C$ and the commutator relation comes from \S\ref{subsection:weights}, we see that elements of $T_m \mathcal M$ acting on $\mathbb V$ have rank either $0$, $2^{k-3}$, or $2^{k-2}$. In particular, the minimum nonzero rank is $\geq 2^{k-3}$, and hence $r\geq 2^{k-1}$. But the weight $0$ and weight $1$ spaces of the spin representation both have dimension $2^{k-2}$, so that $ \dim H^1(C, E_0^\rho)$ and
$ \dim H^0 (C, \widehat{E}_0^\rho) \otimes \omega_C (D)) $ are each equal to $2^{k-2}$ and hence $\leq 2r$. Since $g \geq 2r+2 > 3$, this contradicts \autoref{lemma:rr-lower-bound}.\end{proof}
\begin{lemma}\label{lemma:eca-2}Let $k$ be a natural number. Let $a_1,\dots, a_k$ and $b$ be rational numbers. Suppose that either:

\begin{enumerate}

\item $k\geq 1$ and all sums of the form $\frac{\pm a_1 \pm a_2 \pm \cdots \pm a_k}{2} + \frac{kb}{2}$ are equal to $0$ or $1$, with both values attained.
\item $k > 4$ and either all sums of the form $\frac{\pm a_1 \pm a_2 \pm \cdots \pm a_k}{2} + \frac{kb}{2}$  with an even number of minus signs are equal to $0$ or $1$, with both values attained, or all sums of the form $\frac{\pm a_1 \pm a_2 \pm \cdots \pm a_k}{2} + \frac{kb}{2}$  with an odd minus sign are equal to $0$ or $1$, with both values attained. \end{enumerate}

Then, up to reordering and swapping signs, we have $a_1=1, a_2,\dots, a_k=0$, and $b= \frac{1}{k}$. \end{lemma}

\begin{proof} By swapping signs, we may assume all the $a_i$s are nonnegative.

In case (1), if two of the $a_i$s are nonzero, then without loss of generality we may assume $a_1$ and $a_2$ are both nonzero, and we obtain three distinct sums
	\begin{align*}
	\frac{-a_1 - a_2 +\sum_{i=3}^k a_k + kb}{2}, \frac{-a_1+a_2+\sum_{i=3}^k a_k + kb}{2}, \frac{a_1 + a_2+\sum_{i=3}^k a_k + kb}{2}.
	\end{align*} So all but one of the $a_i$ are $0$, and up to reordering
	only $a_1$ may be nonzero, so the only possible weights are
	$\frac{a_1}{2} + \frac{kb}{2} $ and $\frac{-a_1}{2}+ \frac{kb}{2}$. Since the difference between these must be $1$, we must have $a_1 =1$, which implies $b = \frac{1}{k}$.

	In case (2), if two of the $a_i$s are nonzero, without loss of generality $a_1$ and $a_2$, the sums with an even number of minus signs will include the three distinct values	\begin{align*}
	\frac{a_1 +a_2+ a_3+a_4 + \sum_{i=5}^k a_k +kb}{2}, \\
	\frac{-a_1 +a_2- a_3+a_4+\sum_{i=5}^k a_k+kb}{2}, \\
	\frac{-a_1 -a_2- a_3-a_4+\sum_{i=5}^k a_k +kb}{2}.
	\end{align*}

	Similarly, the sums with an odd number of minus signs will include the three distinct values
	\[\frac{a_1 +a_2+ a_3+a_4 -a_5 + \sum_{i=6}^k a_k+kb}{2},\]
	\[\frac{-a_1 +a_2- a_3+a_4 -a_5 + \sum_{i=6}^k a_k+kb}{2},\]
	\[\frac{-a_1 -a_2- a_3-a_4 -a_5 + \sum_{i=6}^k a_k+kb}{2},\]
	So only one of the $a_i$s may be nonzero. We conclude the argument the same way as in case (1). \end{proof}
	
	\begin{lemma}\label{lemma:lin-alg-2} Let $\nu\geq 7$ be a natural
	number. Let $w \in \mathfrak{so}_\nu$ be an element with a
one-dimensional $1$-eigenspace, a one-dimensional $-1$-eigenspace, and a
$(\nu-2)$-dimensional $0$ eigenspace. Let $N \in \mathfrak{so}_\nu$ be an element such that $w N- N w =-N$. Then if $\nu=2k-1$ is odd, the action of $N$ on the $2^{k-1}$-dimensional spin representation of $\mathfrak{so}_\nu$ has rank either $2^{k-2}$, $2^{k-3}$ or $0$. If $\nu=2k$ is even, the action of $N$ on either of the $2^{k-1}$-dimensional half-spin representations of $\mathfrak{so}_\nu$ has rank either $2^{k-2}$, $2^{k-3}$, or $0$. \end{lemma}
	
\begin{proof}The action of $N$ sends the one-dimensional weight $1$-eigenspace
	to the $(\nu-2)$-dimensional $0$-eigenspace, and sends the
	$0$-eigenspace to the $-1$-eigenspace. Thus $N$ is determined by its
	action on a generator of the $1$-eigenspace space as the condition that
	it is equal to minus its transpose will then determine its action on the
	$0$-eigenspace. Fix a two-dimensional subspace $W$ of the $0$-eigenspace where the quadratic form is nondegenerate. Then $W$ contains nontrivial elements where the quadratic form attains any fixed value, including zero, and thus $W$ intersects each orbit of $\on{SO}_{\nu-2}$ on the $0$-eigenspace.
	
 Hence any element of the $0$-eigenspace is $\on{SO}_{\nu-2}$-conjugate to an element of $W$, and thus $N$ is conjugate to an element that sends the generator of the $1$-eigenspace to an element of $W$. It follows from $N = - N ^T$ that $N$ sends every element of the $0$-eigenspace perpendicular to $W$ to the zero element of the $-1$-eigenspace, i.e. $N$ sends the orthogonal complement of $W$ in the $0$-eigenspace to $0$. Elements of $\mathfrak{so}_\nu$ sending the orthogonal complement of $W$ in the $0$-eigenscace to $0$ form a Lie algebra, isomorphic to $\mathfrak{so}_4$, so we conclude that $N$ is conjugate to a nilpotent element of $\mathfrak{so}_4$.

The restriction to $\mathfrak{so}_4$ of the spin representation of $\mathfrak{so}_\nu$ for $\nu=2k-1$ is odd is isomorphic to the sum of $2^{k-3}$ copies of each half-spin representation of $\mathfrak{so}_4$, and the same is true for the restriction to $\mathfrak{so}_4$ of the half-spin representation of $\mathfrak{so}_\nu$ for $\nu=2k$. This may be checked by comparing weights, since a representation of a simple Lie group is uniquely determined by its weight multiplicities. 
Since $\mathfrak{so}_4$ is isomorphic to $\mathfrak{sl}_2 \times \mathfrak{sl}_2$, with each spin representation the standard representation of one factor, every nilpotent element is a pair of two nilpotent elements in $\mathfrak{sl}_2$, each of which may be zero. The rank of the element is $0$ if both elements of the pair vanish, $2^{k-3}$ if one element is nonvanishing, or $2^{k-2}$ if both are nonvanishing.
\end{proof}

\subsection{Proof of main theorems}
\label{subsection:main-proofs}

We next give the proof of \autoref{theorem:main-thm-2}.
Recall this theorem says that, under suitable hypotheses on $n,g,$ and $r$, 
the monodromy map $R_{\varphi, \rho}:
\on{Mod}_\varphi \to \gl\left( W_1 H^1(\Sigma_{g,n}, \mathbb V^\rho) \right)$ has
image with Zariski closure
$\on{SO}\left( W_1 H^1(\Sigma_{g,n}, \mathbb V^\rho) \right)$ when $\rho$ is
	symplectically self-dual,
$\on{Sp}\left( W_1 H^1(\Sigma_{g,n}, \mathbb V^\rho) \right)$ when $\rho$ is
	orthogonally self-dual,
and the product of
$\on{SL}\left( W_1 H^1(\Sigma_{g,n}, \mathbb V^\rho) \right)$ with a finite
	central subgroup when $\rho$ is not self-dual.

\begin{proof}[Proof of \autoref{theorem:main-thm-2}]
As in notation \autoref{notation:versal-family-of-covers}, let $\pi:
\mathscr{C}\to\mathscr{M}$ be a versal family of $n$-pointed curves of genus
$g$, with associated punctured versal family $\pi^\circ: \mathscr{C}^\circ\to
\mathscr{M}$, and suppose $f: \mathscr{X}\to \mathscr{C}$ gives a versal family
of $H$-covers. Let $\rho: H\to \on{GL}_r(\mathbb{C})$ be an irreducible
representation of $H$, we let $\mathbb V^\rho$ denote the associated local
system on $\Sigma_{g,n}$ and let $\mathbb{U}^\rho$ be the associated local system
on $\mathscr{C}^\circ$; we wish to analyze the connected monodromy group of
$R^1\pi_*^\circ \mathbb{U}^\rho$. By the discussion of \autoref{subsection:moduli-preliminaries}, this suffices.

By \autoref{theorem:simplicity-of-VHS}, the connected monodromy group associated to $W_1
R^1 \pi_*^\circ \mathbb U^\rho$
is a simple group, acting irreducibly.
As described in 
\autoref{subsection:monodromy-possibilities}, there are several
possibilities for the monodromy representation.
Recall we are assuming $g \geq 2r+2$, so the bounds on $g$ in 
\autoref{lemma:no-trivial-rep},
\autoref{lemma:wedge-bound},
and
\autoref{lemma:spin-bound} are satisfied.
By removing the possibilities excluded via
\autoref{lemma:no-trivial-rep},
\autoref{lemma:wedge-bound},
and
\autoref{lemma:spin-bound}, we see that the monodromy representation must act
via the standard representation of $\sl_\nu, \so_\nu$, or $\sp_\nu$, i.e. must be $\on{SO}(W_1H^1(\Sigma_{g,n}, \mathbb V^\rho))$, $\on{Sp}(W_1H^1(\Sigma_{g,n}, \mathbb V^\rho))$, or $\on{SL}(W_1H^1(\Sigma_{g,n}, \mathbb V^\rho))$.

We next show that the identity component of the monodromy group is
\begin{enumerate}
\item $\on{SO}(W_1H^1(\Sigma_{g,n}, \mathbb V^\rho))$ if $\rho$ is symplectically self-dual,
\item $\on{Sp}(W_1H^1(\Sigma_{g,n}, \mathbb V^\rho))$ if $\rho$ is orthogonally self-dual, and
\item $\on{SL}(W_1H^1(\Sigma_{g,n}, \mathbb V^\rho))$ if $\rho$ is not self-dual.
\end{enumerate}
We first handle the case that $\rho$ is self-dual.

If $\rho$ is symplectically self-dual, the antisymmetric pairing on $\rho$
induces a symmetric pairing on
$W_1H^1(\Sigma_{g,n}, \mathbb V^\rho)$, which implies the connected monodromy group must be contained in 
$\on{SO}(W_1H^1(\Sigma_{g,n}, \mathbb V^\rho))$.
This implies the monodromy cannot be
$\on{Sp}(W_1H^1(\Sigma_{g,n}, \mathbb V^\rho))$ or
$\on{SL}(W_1H^1(\Sigma_{g,n}, \mathbb V^\rho))$
so must be 
$\on{SO}(W_1H^1(\Sigma_{g,n}, \mathbb V^\rho))$.
If $\rho$ is orthogonally self-dual, the 
monodromy must be contained in $\on{Sp}(W_1H^1(\Sigma_{g,n}, \mathbb V^\rho))$,
and we similarly obtain it must be equal to $\on{Sp}(W_1H^1(\Sigma_{g,n},
\mathbb V^\rho))$.

To conclude the calculation of the identity component, it remains to show that the identity component of the
Zariski closure of the image of monodromy is
$\on{SL}(W_1H^1(\Sigma_{g,n}, \mathbb V^\rho))$
when $\rho$ is not self dual.
As explained above, there are only three possibilities, and hence it remains to
show the monodromy cannot be 
$\on{SO}(W_1H^1(\Sigma_{g,n}, \mathbb V^\rho))$ or
$\on{Sp}(W_1H^1(\Sigma_{g,n}, \mathbb V^\rho))$. If the monodromy has this form,
then there is an isomorphism of local systems between $W_1H^1(\Sigma_{g,n},
\mathbb V^\rho)$ and $W_1H^1(\Sigma_{g,n}, \mathbb V^\rho)^\vee$, which by
Poincar\'e duality is $W_1H^1(\Sigma_{g,n}, \mathbb V^{\rho^\vee})$. It follows from \autoref{corollary:local-systems-isomorphic} 
that $\rho$ and $\rho^\vee$ are conjugate, contradicting the assumption that $\rho$ is not self-dual.

Having described the connected monodromy group, we now describe the monodromy group itself. If $\rho$ is orthogonally self-dual, Poincar\'e duality gives a symplectic form on $W_1H^1(\Sigma_{g,n}, \mathbb V^\rho)$, so the monodromy group is contained in $\on{Sp}(W_1H^1(\Sigma_{g,n}, \mathbb V^\rho))$ and contains $\on{Sp}(W_1H^1(\Sigma_{g,n}, \mathbb V^\rho))$, and thus must equal $\on{Sp}(W_1H^1(\Sigma_{g,n}, \mathbb V^\rho))$.

Similarly, if $\rho$ is symplectically self-dual, Poincar\'e duality gives a
symmetric form on $W_1H^1(\Sigma_{g,n}, \mathbb V^\rho)$, so the monodromy group
is contained in $\on{O}(W_1H^1(\Sigma_{g,n}, \mathbb V^\rho))$ and contains
$\on{SO}(W_1H^1(\Sigma_{g,n}, \mathbb V^\rho))$, and thus must equal either
$\on{SO}(W_1H^1(\Sigma_{g,n}, \mathbb V^\rho))$ or $\on{O}(W_1H^1(\Sigma_{g,n},
\mathbb V^\rho))$. However, we now check that the monodromy group is not $\on{O}(W_1H^1(\Sigma_{g,n}, \mathbb V^\rho))$: The representation $\rho$, being symplectic and unitary, necessarily has the structure as a representation over the quaternions $\mathbb H$, i.e. has an $\mathbb R$-linear action of the quaternions compatible with the action of $H$. 
Hence the complex local system $W_1H^1(\Sigma_{g,n}, \mathbb V^\rho)$ has an
$\mathbb R$-linear action of the quaternions compatible with the mapping class
group action. Thus, for each element $\sigma$ in the mapping class group, the
eigenspace with eigenvalue $1$ of $\sigma$, acting on
$W_1H^1(\Sigma_{g,n},\mathbb V^\rho)$, has an action of the quaternions and
hence is an even-dimensional complex vector space. However, any elements of
$\on{O}_{N} - \on{SO}_N$ for even $N$ have an odd-dimensional $1$-eigenspace, so the image of the mapping class group is contained in $\on{SO}(W_1H^1(\Sigma_{g,n}, \mathbb V^\rho))$ and thus the monodromy group must be $\on{SO}(W_1H^1(\Sigma_{g,n}, \mathbb V^\rho))$.

Finally, if $\rho$ is not self-dual, the monodromy group is a subgroup of $\on{GL}(W_1H^1(\Sigma_{g,n}, \mathbb V^\rho))$ whose identity component is $\on{SL}(W_1H^1(\Sigma_{g,n}, \mathbb V^\rho))$ and hence must be the product of $\on{SL}(W_1H^1(\Sigma_{g,n}, \mathbb V^\rho))$ 
with a finite subgroup of the center of $\on{GL}(W_1H^1(\Sigma_{g,n}, \mathbb V^\rho))$.\end{proof}

For the proof of \autoref{theorem:main-thm-1}, we will also need the following
explicit description of the symplectic centralizer.
\begin{lemma}
	\label{lemma:symplectic-centralizer}
	Let $\rho_1, \ldots, \rho_\alpha$ denote the orthogonally self-dual
	complex irreducible representations of $H$, let $\rho_{\alpha+1},\ldots,
	\rho_{\alpha+\beta}$ denote the symplectically self-dual irreducible
	complex representations of $H$, and let
	$(\rho_{\alpha+\beta+1},\rho_{\alpha+\beta+\gamma+1}), \ldots,
	(\rho_{\alpha+\beta+\gamma},\rho_{\alpha+\beta+2\gamma})$ denote the
	dual pairs of complex irreducible representations of $H$.
	We use $\mathbb V^\rho$ to denote the local system on $\Sigma_{g,n}$
	corresponding to $\rho$.
	Then the isomorphism
	\[  H^1 ( \Sigma_{g'}, \mathbb C) \overset{\sim}{\longrightarrow}
	\prod_{i=1}^{\alpha+\beta+2\gamma}    \rho_i^\vee \otimes
W_1H^1(\Sigma_{g,n},\mathbb V^{\rho_i})\]
		induces an isomorphism
	\begin{align}
		\label{equation:symplectic-centralizer}
	\hspace{-.5in} \on{Sp}( H^1 ( \Sigma_{g'}, \mathbb C))^H\overset{\sim}{\longrightarrow}
\prod_{i=1}^\alpha
\on{Sp}(W_1H^1(\Sigma_{g,n}, \mathbb V^{\rho_i})) \times \prod_{i=\alpha+1}^{\alpha+\beta}
\on{O}(W_1H^1(\Sigma_{g,n}, \mathbb V^{\rho_i})) \times \prod_{i =
\alpha+\beta+1}^{\alpha+\beta+\gamma} \on{GL}( W_1H^1(\Sigma_{g,n},\mathbb
V^{\rho_i})).
	\end{align}
Hence the commutator of 
$\on{Sp}( H^1 ( \Sigma_{g'}, \mathbb C))^H$ is 
\begin{equation}\label{group-to-contain} 
	\prod_{i=1}^\alpha
	\on{Sp}(W_1H^1(\Sigma_{g,n}, \mathbb V^{\rho_i})) \times \prod_{i=\alpha+1}^{\alpha+\beta}
	\on{SO}(W_1H^1(\Sigma_{g,n}, \mathbb V^{\rho_i})) \times \prod_{i =
\alpha+\beta+1}^{\alpha+\beta+\gamma} \on{SL}( W_1H^1(\Sigma_{g,n},\mathbb
V^{\rho_i})). \end{equation} 
\end{lemma}
The above lemma follows from the explicit description of the symplectic centralizer given in
\cite[Theorem 3.1.10]{jain:big-mod-ell-monodromy} (which seems to implicitly
work over finite fields, but the same proof works over the complex
numbers).

We next prove our main result, \autoref{theorem:main-thm-1},
which states that under suitable hypotheses on $g,n$ and the maximal dimension $\overline{r}$
of an irreducible
representation of $H$, the identity component of the Zariski
closure of the monodromy map $R_\varphi : \on{Mod}_\varphi \to \on{Sp}\left(
H^1(\Sigma_{g'}, \mathbb C) \right)^H$ is the commutator subgroup of 
$\on{Sp}\left( H^1(\Sigma_{g'}, \mathbb C) \right)^H$.

\begin{proof}[Proof of \autoref{theorem:main-thm-1}] 
%
Let $G$ be the identity component of the Zariski closure of the image of the
mapping class group in \eqref{equation:symplectic-centralizer}. In other words, $G$ is the connected monodromy
group of $\bigoplus_{i=1}^{\alpha+\beta+\gamma} W_1H^1(\Sigma_{g,n}, \mathbb
V^{\rho_i})$.
Checking \autoref{theorem:main-thm-1}, i.e. that the virtual image of the
mapping class group is Zariski dense in the commutator subgroup of this group,
is equivalent to checking that $G$ contains
\eqref{group-to-contain}. By the discussion of \autoref{subsection:moduli-preliminaries}, we may interpret all the representations in question as monodromy representations associated to local systems on the base $\mathscr{M}$ of a versal family of $\varphi$-covers.

To check $G$ contains 
\eqref{group-to-contain}, we apply the Goursat-Kolchin-Ribet criterion of Katz
\cite[Proposition 1.8.2]{katz-esde}. We let $V_i$ be the representation of $G$
acting on $W_1H^1(\Sigma_{g,n}, \mathbb V^{\rho_i})$. Let $G_i$ be the image of $G$ in
$\on{GL}(V_i)$, which we know from \autoref{theorem:main-thm-2} is
$\on{Sp}(W_1H^1(\Sigma_{g,n}, \mathbb V^{\rho_i}))$ for $i$ from $1$ to $\alpha$,
$\on{SO}(W_1H^1(\Sigma_{g,n}, \mathbb V^{\rho_i}))$ for $i$ from $\alpha+1$ to $\alpha + \beta$, and $\on{SL}(
W_1H^1(\Sigma_{g,n},\mathbb V^{\rho_{i}}))$ for $i$ from $1+\alpha+\beta$ to
$\alpha+\beta+\gamma$. Then \cite[Proposition 1.8.2]{katz-esde} guarantees that
$G^{0, \on{der}} = \prod_{i=1}^{\alpha+\beta+\gamma} G_i^{0, \on{der}}$, which
is the desired \eqref{group-to-contain}, as long as four conditions are
satisfied, which we verify next.

The first condition is that for each $i$, $G_i^{0, \on{der}}$ operates irreducibly on $V_i$, and its Lie algebra is simple. \autoref{theorem:main-thm-2} guarantees that the action is by the standard representation, which is irreducible except in the case of the two-dimensional standard representation of $\on{SO}_2$, and the Lie algebra is simple except in the case of the four-dimensional standard representation of $\on{SO}_4$. But our assumptions on $r$ and $g$ imply each representation has dimension $\geq 2r_i (g-1) \geq 2r_i(2r_i+1) \geq 6$ where $r_i = \dim \rho_i$, so this condition is always satisfied.

The second condition is that for any $i\neq j$, $(G_i^{0, \on{der}},V_i)$ and $(G_j^{0, \on{der}},V_j)$ are Goursat-adapted in the sense of \cite[\S1.8]{katz-esde}, but this follows by \cite[Example 1.8.1]{katz-esde} from the fact that $G_i$ is a classical group and $V_i$ is its standard representation of dimension $\geq 6$ with the possible exception that $G_i  \cong G_j \cong SO_8$. However, we know from \autoref{theorem:main-thm-2} that $G_i$ is $\on{SO}_n$ only if $\rho_i$ is symplectically self-dual, which implies $\rho_i$ has rank $r_i \geq 2$ since all symplectically self-dual representations are even-dimensional and thus $\dim V_i \geq 2r_i  (2r_i+1) \geq 20>8$, so $G_i$ cannot be $\on{SO}_8$.

The third and fourth condition say that for each $i\neq j$ and each character
$\chi$ of $G$, neither the representation $V_i$ nor its dual is isomorphic to
$V_j\otimes \chi$. Since $G$ is the connected monodromy group of a local system
of geometric origin, it is necessarily simple, and so $\chi$ is finite-order.
Thus by passing to a finite cover of $\mathscr{M}$, $\chi$ becomes trivial,
showing that $W_1H^1(\Sigma_{g,n},\mathbb V^{\rho_i})$ is isomorphic to
$W_1H^1(\Sigma_{g,n},\mathbb V^{\rho_j})$ or its dual over this finite covering. This case
is ruled out by \autoref{corollary:local-systems-isomorphic} unless $\rho_i$ is
conjugate to $\rho_j$ or its dual, which is impossible as $i\neq j$ and $i,j\leq
\alpha+\beta+\gamma$ so they cannot be part of a dual pair.

Since all the conditions of the 
Goursat-Kolchin-Ribet criterion of Katz
\cite[Proposition 1.8.2]{katz-esde}
are satisfied, $G^{0, \on{der}} =
\prod_{i=1}^{\alpha+\beta+\gamma} G_i^{0, \on{der}}$ and thus $
\prod_{i=1}^{\alpha+\beta+\gamma} G_i^{0, \on{der}}$ is a subgroup of $G$, as desired.
\end{proof}

\subsection{Proofs of corollaries}\label{subsec:proof-of-cors}

We next prove \autoref{corollary:mumford-tate},
which states that, for $X$ a very general $H$-curve, under suitable hypotheses
on $g$ and $H$, the Mumford-Tate
group of $H^1(X, \mathbb Q)$ contains the commutator subgroup of $\on{Sp}(H^1(X,
	\mathbb Q))^H$ and is contained in $\on{GSp}(H^1(X,
	\mathbb Q))^H$
\begin{proof}[Proof of \autoref{corollary:mumford-tate}]
This is immediate from Andr\'e's theorem of the fixed part \cite[Theorem 1 on p.
10]{AndreFixed} and \autoref{theorem:main-thm-1}; Andr\'e's theorem implies the
generic Mumford-Tate group contains the monodromy group, namely the commutator
subgroup of $\on{Sp}(H^1(\Sigma_{g'}, \mathbb{C}))^H$. On the other hand, it is
contained in the centralizer of $H$ in $\on{GSp}(H^1(\Sigma_{g'}, \mathbb{C}))$,
as it centralizes $H$ and preserves the symplectic pairing on $H^1$ up to
scaling (as the symplectic pairing corresponds to a 
a Hodge class). 
\end{proof}
We next prove \autoref{corollary:jacobian-aut}, which, under the same hypotheses
as in the previous corollary, states that the endomorphism algebra of the
Jacobian of a very general  $H$-curve $X$ is $\mathbb Q[H]$.

\begin{proof}[Proof of \autoref{corollary:jacobian-aut}]
Let $G$ be the Mumford-Tate group of $H^1(X, \mathbb{Q})$. It suffices to show that the natural map $\mathbb{Q}[H]\to \on{End}_{\on{HS}}(H^1(X, \mathbb{Q}))=\on{End}_G(H^1(X, \mathbb{Q}))$ is a bijection, where $\on{HS}$ is the category of $\mathbb{Q}$-polarizable variations of Hodge structure. As $G$ contains the commutator subgroup $S$ of $\on{Sp}(H^1(X, \mathbb{Q}))^H$ by \autoref{corollary:mumford-tate}, it moreover suffices to show that the map 
	 $$\mathbb{Q}[H]\to \on{End}_{S}(H^1(X,
	 \mathbb{Q}))$$ is a bijection; we may do so after tensoring with
	 $\mathbb{C}$, whence $\mathbb{C}[H]=\prod_{\rho_i} \on{End}(\rho_i)$,
	 where the $\rho_i$ run over the irreducible complex representations of $H$.
	 Similarly, 
	for $\mathbb V^{\rho_i}$ the local system on $X/H-D$ associated to $\rho_i$
	 $$H^1(X, \mathbb{C})=\bigoplus \rho_i\otimes W_1H^1(X/H-D,
	 \mathbb V^{\rho_i}),$$ where $D\subset X/H$ is the branch locus of the natural map
	 $X\to X/H$. By the description of the symplectic centralizer,
	 \autoref{lemma:symplectic-centralizer}, the $W_1H^1(X/H-D,
	 \rho_i)$ are simple and pairwise non-isomorphic as $S$-representations. Hence $$\on{End}_{S}(H^1(X, \mathbb{C}))=\prod_{\rho_i} \on{End}_{\mathbb{C}}(\rho_i)$$ has dimension $|H|$. As the map we are studying is evidently injective, it is necessarily surjective as well by a dimension count.
\end{proof}

\section{Proof of \autoref{theorem:large-n}: big monodromy for $n$ large}
\label{section:large-n}
We now prove \autoref{theorem:large-n}. We will require the following slight
strengthening of \cite[Theorem 1.3.4]{landesmanL:geometric-local-systems}, which
was not quite stated optimally. See \cite[Definitions 1.2.1 and
1.2.3]{landesmanL:geometric-local-systems} for the definitions of hyperbolic
curve and analytically very general.

\begin{theorem}
	\label{theorem:hn-constraints}
	Let $(C,D)$ be hyperbolic of genus $g$ and let
	$({E}, \nabla)$ be a flat
	vector bundle on $C$ with regular singularities along $D$,
	and irreducible
monodromy.
Suppose
 $(E',\nabla')$ 
is an isomonodromic deformation
of $({E}, \nabla)$ to an analytically general nearby curve,
with Harder-Narasimhan filtration $0 = (F')^0 \subset (F')^1 \subset \cdots \subset
(F')^m =
E'$. For $1 \leq i \leq m$, let $\mu_i$ denote the slope of
$\on{gr}^{i}_{HN}E' := (F')^i/(F')^{i-1}$.
Suppose $E'$ is not semistable. Then for every $0 < i < m$, there
	exists $j < i < k$ with $$\rk \on{gr}^{j+1}_{HN}E'\cdot \rk
	\on{gr}^k_{HN}E'\geq g+1$$ and 
 $$0<\mu_{j+1}-\mu_{k}\leq 1.$$
\end{theorem}
\begin{proof}
In fact this is precisely the output of the proof of 	\cite[Theorem 1.3.4]{landesmanL:geometric-local-systems}.
\end{proof}
\begin{corollary}\label{corollary:mu-diff}
	With notation as in \autoref{theorem:hn-constraints}, let
	$\mu_{\operatorname{diff}}=\mu_1-\mu_m$. Then $$\mu_{\operatorname{diff}}\leq \frac{3\on{rk}(E)}{2\sqrt{g+1}}+3.$$
\end{corollary}
\begin{proof}
Set $s$ to be the greatest
integer which is strictly less than $\mu_{\on{diff}}/3$.
We first show that $$\on{rk}(E)\geq 2\sqrt{g+1}\cdot s.$$ 

Fix disjoint closed intervals $I_1, \cdots, I_s$ in $[\mu_m, \mu_1]$, each of
length $3$. For each interval, $I_t=[a_t, a_t+3]$, there exists $i_t$ with
$\mu_t\in  [a_t+1, a_t+2]$, by \autoref{theorem:hn-constraints}. By the same theorem, there exists $j_t<i_t<k_t$ with 
\begin{equation}
\rk \on{gr}^{j_t+1}_{HN}E'\cdot \rk
	\on{gr}^{k_t}_{HN}E'\geq g+1
	\label{equation:graded-am-gm}
\end{equation}
and 
$$0<\mu_{j_t+1}-\mu_{k_t}\leq 1.$$ In particular, $\mu_{j_t+1}, \mu_{k_t}\in I_t$, and hence all the integers $j_t+1, k_t$ are distinct.
 
Now we have $$\on{rk}(E)=\on{rk}(E')\geq \sum_{t=1}^s \left(\on{gr}^{j_t+1}_{HN}E'+
\on{gr}^{k_t}_{HN} E'\right)$$ which is bounded below by $(2\sqrt{g+1})\cdot s$ by the
AM-GM inequality and \eqref{equation:graded-am-gm}.
So $s \leq \frac{\rk E}{2 \sqrt{g+1}}$,
which implies
 \begin{equation*}
 \mu_{\operatorname{diff}}  \leq 3s+3
  \leq \frac{3\on{rk}(E)}{2\sqrt{g+1}}+3. \qedhere
 \end{equation*}
\end{proof}
\begin{lemma}\label{lemma:degree-bound}
	With notation as in \autoref{notation:rep-to-vector-bundle}, one of $E^\rho_0$ and $E^{\rho^\vee}_0$ has slope less than or equal to $-\frac{\Delta}{2r}$.
\end{lemma}
\begin{proof}
It suffices to show that $E^\rho_0\oplus E^{\rho^\vee}_0=E^{(\rho\oplus\rho^\vee)}_0$ has degree less than or equal to $-\Delta$. But if $\lambda$ is an eigenvalue of local monodromy of $\rho$ at a point $x$ of $D^\rho_{\text{non-triv}}$, then $\lambda^{-1}$ is an eigenvalue of local monodromy of $\rho^\vee$ at $x$. Hence if $\alpha\neq 0$ is a parabolic weight of $E^\rho_\star$ at $x$, then $1-\alpha$ is a parabolic weight of $E^{\rho^\vee}_\star$ at $x$. In particular, the sum of the parabolic weights of $E^{(\rho\oplus\rho^\vee)}_\star$ at $x$ is at least $1$ for each $x\in D^\rho_{\text{non-triv}}$.

Hence,
\begin{equation*}
\deg E^{(\rho\oplus\rho^\vee)}_0=-\sum_{x_j\in D^\rho_{\text{non-triv}}}\sum_{i=1}^{n_j}
	\alpha^i_j \dim(E_j^i/E_j^{i+1})\leq  -\sum_{x\in
	D^\rho_{\text{non-triv}}} 1\leq  -\Delta.
	\qedhere
\end{equation*}
\end{proof}

\begin{lemma}\label{lemma:large-n-gg}
With notation as in \autoref{notation:rep-to-vector-bundle}, suppose $(C,D)$ is a general $n$-pointed curve. If 
\begin{align}
	\label{equation:delta-constraint}	
\Delta>\frac{3r^2}{\sqrt{g+1}}+8r
\end{align}
then at least one of $(E^\rho_0)^\vee\otimes \omega_C$ and $(E^{\rho^\vee}_0)^\vee\otimes \omega_C$ is globally generated.
\end{lemma}
\begin{proof}
	Without loss of generality we may (by replacing $\rho$ with $\rho^\vee$ if necessary) assume $$\mu(E^\rho_0)\leq -\frac{\Delta}{2r},$$ by \autoref{lemma:degree-bound}. We will show that in this case $(E^\rho_0)^\vee\otimes \omega_C$ is globally generated.
	
	Let $\mu_1>\cdots>\mu_m$ be the set of slopes of the graded pieces of
	the Harder-Narasimhan filtration of $E^\rho_0$, as in
	\autoref{theorem:hn-constraints}, and let
	$\mu_{\operatorname{diff}}=\mu_1-\mu_m$. By \autoref{corollary:mu-diff}, we have
	$$\mu_{\operatorname{diff}}\leq \frac{3r}{2\sqrt{g+1}}+3.$$ Hence $$\mu_1\leq
	-\frac{\Delta}{2r}+\mu_{\operatorname{diff}} \leq -\frac{\Delta}{2r}+\frac{3r}{2\sqrt{g+1}}+3<-1$$ 
	by 
	rearranging the assumption
	\eqref{equation:delta-constraint}. Thus the Harder-Narasimhan slopes of $(E^\rho_0)^\vee\otimes \omega_C$ are all greater than $2g-1$.
	
	We claim any vector bundle $V$ so that each graded piece of its Harder-Narasimhan
	filtration has slope more than $2g - 1$ is globally
	generated. This is well known, but we explain it for completeness. 
	A vector bundle $V$ is globally generated at $p$ if
	$H^1(C,V(-p)) = 0$,
	as then $H^0(C, V) \to H^0(C, V|_p)$ is surjective.
	By \cite[Lemma 6.3.5]{landesmanL:geometric-local-systems},
	if $W$ is a vector bundle so that each graded piece
	of its Harder-Narasimhan
	filtration has slope more than $2g - 2$, $H^1(C, W) = 0$. Therefore,
	$H^1(C, V(-p)) = 0$ for any point $p$, so $V$ is globally generated.
	This shows $({E}^\rho_0)^\vee \otimes \omega_C$ is globally generated.
\end{proof}

\subsection{Proof of \autoref{theorem:large-n}}
\label{subsection:large-n}
Recall that we are aiming to prove that, once $\Delta$ is sufficiently large,
(larger than $\frac{3r^2}{\sqrt{g+1}} + 8r$,)
there are no nonzero vectors with finite orbit in $W_1 H^1(\Sigma_{g,n}, \mathbb
V^\rho)$ under the image of $\on{Mod}_{\varphi}$; here, $\on{Mod}_{\varphi}$ is
the stabilizer of $\varphi: \pi_1(\Sigma_{g,n}) \to H$ in $\on{Mod}_{g,n+1}$.

As the virtual representations of $\on{Mod}_{g,n+1}$ on $$W_1H^1(\Sigma_{g,n},
\mathbb V^\rho), W_1H^1(\Sigma_{g,n}, \mathbb V^{\rho^\vee})$$ are semisimple and dual to one another,
it suffices to prove this for one of $\rho$ and $\rho^\vee$, so we may without loss of generality assume by \autoref{lemma:large-n-gg} that $(E^\rho_0)^\vee \otimes \omega_C$ is globally generated for $(C,D)$ a general $n$-pointed curve.
A similar argument to that given in \cite[Proposition 3.4]{landesmanL:applications-putman-wieland}
shows that the virtual action of $\on{Mod}_{g,n+1}$ on
$\on{GL}(W_1H^1(\Sigma_{g,n}, \mathbb V^\rho))$ has no nonzero finite-orbit vectors.

To make this proof slightly more self contained, we recall briefly the idea 
of the proof of 
\cite[Proposition 3.4]{landesmanL:applications-putman-wieland}.
Namely, the derivative of the period map as described in
\autoref{proposition:derivative-of-period-map}
associated to the Hodge filtration of
$W_1H^1(\Sigma_{g,n}, \mathbb V^\rho)$
can be identified with a map
\begin{equation*}
	 H^0(C, \widehat
	E^\rho_0\otimes \omega_C(D))\otimes H^0(C, (E^\rho_0)^\vee\otimes \omega_C)\to H^0(C, \omega_C^{\otimes 2}(D))\end{equation*}
If there is a vector with finite orbit in $W_1H^1(\Sigma_{g,n}, \mathbb V^\rho)$ then the adjoint map $$H^0(C, \widehat
	E^\rho_0\otimes \omega_C(D))\to \operatorname{Hom}(H^0(C, (E^\rho_0)^\vee\otimes \omega_C), H^0(C, \omega_C^{\otimes 2}(D)))$$ has a non-zero kernel. 
A vector in the kernel yields a nonzero map 
$$\phi: (E^\rho_0)^\vee\otimes \omega_C\to \omega_C^{\otimes 2}(D)$$
inducing the $0$ map on global sections.
Then, $\ker \phi \subset  (E^\rho_0)^\vee\otimes \omega_C$ would
be a proper subbundle inducing an isomorphism on global sections, implying 
$(E^\rho_0)^\vee\otimes \omega_C$ is not generically globally generated, hence not globally generated,
a contradiction.
\qed

\section{Big monodromy for Kodaira fibrations}
\label{section:kodaira}
In this section we analyze the connected monodromy groups of certain smooth
proper families $$\pi: S\to Z$$ over $\mathbb C$, referred to as \emph{Kodaira-Parshin fibrations}, where $Z$ is a smooth curve and $\pi$ is smooth and proper of relative dimension one, with connected fibers. Loosely speaking, these families will parameterize covers of a fixed curve $C$, branched at a moving point.
We next give a few definitions to fix notation for Kodaira-Parshin fibrations.
\begin{definition}\label{definition:kp-fibration}
	Let $Z$ be a smooth, not necessarily proper, connected curve. A smooth
	proper morphism $\pi: S\to Z$ of relative dimension one with connected
	fibers is a \emph{Kodaira-Parshin fibration} if there exists a smooth
	proper curve $C$, a reduced effective divisor $D\subset C$, a nonconstant map $f: Z\to C$, and a dominant finite
	map $q: S\to Z\times C$, branched only over the graph of $f$ and $Z \times D$, such that
	$\pi$ factors as $\pi = \pi_1 \circ q$,
	where $\pi_1$ is projection onto the first coordinate, as pictured in
	the following diagram
	\begin{equation}
		\label{equation:}
		\begin{tikzcd} 
			S \ar[swap] {rd}{\pi} \ar {r}{q} & Z \times C \ar {d}{\pi_1}
			\ar{r}{\pi_2} & C \\
			& Z. &
	\end{tikzcd}\end{equation}
\end{definition}
\begin{definition}
	\label{definition:topological-type}
	Continuing with notation as in \autoref{definition:kp-fibration},
	fix $z\in Z$, and consider the map $\pi^{-1}(z)\to C$ given as $\pi_2\circ q|_{\pi^{-1}(z)}$, where $\pi_2: Z\times C\to C$ is projection on to the second coordinate. 
	We refer to the underlying map on topological spaces (in the Euclidean topology) as the \emph{topological type} of the family $\pi$. 
\end{definition}
\begin{remark}
	\label{remark:}
	Note that the topological type of $\pi$, as defined in
	\autoref{definition:topological-type} is independent of $z$, up to homeomorphism, by Ehresmann's theorem.
	That is, every fiber of $\pi$ is a cover of $C$ of the same topological type.
	Hence it makes sense to speak of the topological type of $\pi$, and not
	just of $\pi$ over $z$.
\end{remark}
\begin{definition}
	\label{definition:distinguished-class}
	Continuing with notation as in \autoref{definition:topological-type},
the topological type of a Kodaira-Parshin fibration is a ramified map of
	(topological) surfaces $t:\Sigma_{g'}\to \Sigma_g$; we refer to the the
	Galois group of (the Galois closure of) this map as the \emph{Galois
	group} of the Kodaira-Parshin fibration. Note that $\Sigma_g$ has a
	{\em distinguished point}, which corresponds to $f(z)$ under the isomorphism $\Sigma_g
	\simeq C$. We refer to the conjugacy class of the monodromy about this point in the Galois group $H$ of the cover as the \emph{distinguished class} $[h]\subset  H$.
\end{definition}
We next recall the classical Kodaira-Parshin trick; see
\cite[Proposition 7]{parshin:algebraic-curves-over-function-fields-i} for the
original construction, and also
\cite[Proposition 5.1.1]{landesmanL:geometric-local-systems} for a more modern
construction in families. See also \cite{atiyah1969signature} and \cite{kodaira1967certain} for closely related constructions.
\begin{example}[The Kodaira-Parshin trick]
	\label{example:kp-trick}
	Let $C$ be a smooth proper curve, and let $p: S\to C\times C$ be a
	dominant finite map branched only over the diagonal $\Delta$. More
	precisely, $S$
	is the normalization of $C\times C- \Delta$ in the function
field of a finite \'etale cover of $C \times C - \Delta$. Let
$$S\overset{\pi}{\longrightarrow} Z\to C$$ be the Stein factorization of the
composition $\pi_1\circ p$, where $\pi_1: C \times C \to C$ is projection onto
the first factor. Then $\pi$ is a Kodaira-Parshin fibration. 
\end{example}
\begin{remark}
	\label{remark:}
	Define a \emph{Kodaira fibration} to be a surjective smooth proper morphism from a smooth projective surface to a smooth \emph{proper} curve.
	In the setting of \autoref{example:kp-trick}, $C$ itself is proper, so $\pi$ is in fact a Kodaira
fibration.
The construction of \autoref{example:kp-trick} is known as the \emph{Kodaira-Parshin trick}, and is used e.g.~by Faltings in his proof of the Mordell conjecture.
\end{remark}
The main theorem of this section is a computation of the connected monodromy
group of a Kodaira-Parshin fibration of topological type $t:\Sigma_{g'}\to \Sigma_g$, when $g$ is large compared to the dimensions of the irreducible representations of $H$. We give a statement for Kodaira-Parshin fibrations with Galois topological type, but one may easily deduce an analogous statement for arbitrary Kodaira-Parshin fibrations by passing to Galois closures.
To state the theorem, we use the following notation.

\begin{notation}
	\label{notation:representations}
	Let $\pi: S\to Z$ be a Kodaira-Parshin fibration of topological type $t:
\Sigma_{g'}\to \Sigma_g$,  with $t$ a Galois cover branched at $n$ points, with
Galois group $H$ and distinguished class $[h]\subset H$, as defined in \autoref{definition:kp-fibration}. 
Let $\varphi: \pi_1(\Sigma_{g,n},x) \twoheadrightarrow H$ denote the surjection
associated to $t$, for $x$ a chosen basepoint.
Let $\rho_1^{\o}, \cdots, \rho_\alpha^{\o}$ be the irreducible orthogonally self-dual
complex
$H$-representations with $\rho_i^\o([h])$ nontrivial, $\rho_1^\sp, \ldots,
\rho_\beta^\sp$ the irreducible symplectically self-dual $H$ irreducible complex representations
with $\rho_i^\sp([h])$ non-trivial, and let $(\rho^\sl_{1},
\rho^\sl_{\gamma+1}), \ldots, (\rho^\sl_\gamma, \rho^\sl_{2\gamma})$ be the set
of dual pairs of complex irreducible $H$-representations with $\rho^\sl_i([h])$ non-trivial. Let $\overline{s}$ be the maximal dimension of an irreducible representation $\rho$ of $H$ with $\rho([h])$ non-trivial. 
We use $\mathbb V^\rho$ to denote the local system on $\Sigma_{g,n}$ associated
to the composition $\rho \circ \varphi$.
\end{notation}

\begin{theorem}\label{theorem:kodaira-monondromy}
	With notation as in \autoref{notation:representations},
suppose $g> \max(2\overline{s}+1, \overline{s}^2)$.
Then, the identity component of the Zariski-closure of $\pi_1(Z, z)$ in
$\on{GL}(H^1(\pi^{-1}(z), \mathbb{C}))$ is  
\begin{align}
	\label{equation:kodaira}
	\prod_{i=1}^\alpha
	\on{Sp}(W_1H^1(\Sigma_{g,n}, \mathbb V^{\rho_i^\o})) \times \prod_{i=1}^{\beta}
	\on{SO}(W_1H^1(\Sigma_{g,n}, \mathbb V^{\rho_i^\sp})) \times \prod_{i = 1 }^{\gamma}
	\on{SL}( W_1H^1(\Sigma_{g,n},\mathbb V^{\rho_i^\sl})),
\end{align}
the subgroup of the derived subgroup of the centralizer of $H$ in $\on{Sp}(H^1(\Sigma_{g'}, \mathbb{C}))$ corresponding to those $H$-irreps non-trivial on $[h]$.
\end{theorem}

The following purely topological corollary follows immediately from the definitions, as in \autoref{subsection:moduli-preliminaries}---loosely speaking, it says that in the representations considered in \autoref{theorem:main-thm-1}, the restriction to the ``point-pushing subgroup'' still has large image. 

\begin{corollary}\label{corollary:point-pushing-big-image}
With notation as in \autoref{notation:representations},
suppose that $n>0$, and let $h\in \pi_1(\Sigma_{g,n}, x)$ be a loop around a
fixed puncture $z$ of $\Sigma_{g,n}$, 
so that $\varphi(h)= [h]$ is the distinguished class.
Suppose that $g> \on{max}(2\overline{s}+1, \overline{s}^2).$ Let $P_\varphi$ be the stabilizer of $\varphi$ in the kernel of the map $$\on{Mod}_{g, n+1}\to \on{Mod}_{g,n}$$ induced by forgetting $z$. (Here we view $\Sigma_{g,n}$ as an $n+1$-marked surface, with $x$ an additional marked point.)
The identity component of the Zariski closure of $$R_\varphi|_{P_\varphi}: P_\varphi\to \on{Sp}(H^1(\Sigma_{g'}, \mathbb{C}))^H,$$ in
$\on{GL}(H^1(\Sigma_{g'}, \mathbb{C}))$ is  $$\prod_{i=1}^\alpha
\on{Sp}(W_1H^1(\Sigma_{g,n}, \mathbb V^{\rho_i^\o})) \times \prod_{i=1}^{\beta}
\on{SO}(W_1H^1(\Sigma_{g,n}, \mathbb V^{\rho_i^\sp})) \times \prod_{i = 1 }^{\gamma}
\on{SL}( W_1H^1(\Sigma_{g,n},\mathbb V^{\rho_i^\sl})),$$ 
the subgroup of the derived subgroup of the centralizer of $H$ in $\on{Sp}(H^1(\Sigma_{g'}, \mathbb{C}))$ corresponding to those $H$-irreps non-trivial on $[h]$. Here $R_\varphi$ is defined as in \autoref{theorem:main-thm-1}.
\end{corollary}

In order to prove \autoref{theorem:kodaira-monondromy}, we recall some
notation from \autoref{notation:versal-family-of-covers}.
Given a branched cover $$t: \Sigma_{g'}\to \Sigma_g$$ of topological surfaces,
with Galois group $H$ and $n$ branch points (here consisting of the points of $D$ together with the $f(z)$), there is a versal family of covers over a variety $\mathscr{M}_t$, parameterizing branched covers of Riemann surfaces of topological type $t$ (see e.g.~\cite[\S6]{hurwitz}), carrying a family of covers $\mathscr{X}_t/\mathscr{C}_t$, as in the diagram below:
$$\xymatrix{
		\mathscr{X}_t \ar[r]^f \ar[rd]_{\pi'}& \mathscr{C}_t \ar[d]_\pi\\
	&  \mathscr{M}_t. \ar@/_/[u]_{s_1, \cdots, s_n} }$$
After replacing $Z$ with an \'etale cover, our given Kodaira-Parshin fibration $S\to Z$ of topological type $t$ is the pullback of $\pi'$  by a
map $\iota: Z\to \mathscr{M}_t$. Note that the family of pointed curves
$\mathscr{C}_t/\mathscr{M}_t$ induces a dominant map $\mathscr{M}_t\to \mathscr{M}_{g,n}$
with finitely many geometric points in each fiber. 
By definition, $\iota$ dominates a component of the the fiber of the composition
$\mathscr{M}_t\to \mathscr{M}_{g,n}\to \mathscr{M}_{g,n-1}$, given by forgetting
the distinguished point as defined in \autoref{definition:distinguished-class}.

The following lemma is the main ingredient in the proof, and describes the
monodromy of a Kodaira-Parshin family associated to a particular representation.
\begin{lemma}\label{lemma:kodaira-individual-reps}
With notation as above, let $[h]\in H$ be the distinguished class, and let
$\rho: H\to \on{GL}_r(\mathbb{C})$ be an irreducible $H$-representation. Suppose
$g>\max(2r+1, r^2)$. Then the identity component of the Zariski-closure of the monodromy
representation 
\begin{align}
	\label{equation:kp-monodromy-rep}
\pi_1(Z,z)\to \on{GL}(W_1H^1(\Sigma_{g,n}, \mathbb V^\rho))
\end{align}
is trivial if $\rho([h])$ is trivial. If $\rho([h])$ is non-trivial, it is
\begin{enumerate}
	\item $\on{SO}(W_1H^1(\Sigma_{g,n}, \mathbb V^\rho))$ if $\rho$ is symplectically self-dual,
	\item $\on{Sp}(W_1H^1(\Sigma_{g,n}, \mathbb V^\rho))$ if $\rho$ is orthogonally self-dual, or
	\item $\on{SL}(W_1H^1(\Sigma_{g,n}, \mathbb V^\rho))$ if $\rho$ is not self-dual.
\end{enumerate} 	
\end{lemma}
\begin{proof}
	We first observe that if $\rho([h])$ is trivial, the connected component of the monodromy group in question is also trivial. Indeed, in this case, let $K\subset H$ be the kernel of $\rho$; the monodromy representation in question appears in the cohomology of $S/K\to Z$, which is isotrivial. Indeed, the fibers of this map are covers of $C$ branched at a fixed divisor,  and hence do not vary in moduli.

	We now suppose $\rho([h])$ is non-trivial. We first claim that it
	suffices to show that the monodromy representation 
	\eqref{equation:kp-monodromy-rep}
	has infinite image. Indeed, by the discussion above the statement of the
	lemma, \eqref{equation:kp-monodromy-rep} factors through a representation 
	\begin{align}
		\label{equation:moduli-space-monodromy-rep}
	\pi_1(\mathscr{M}_t)\to \on{GL}(W_1H^1(\Sigma_{g,n}, \mathbb V^\rho)).
	\end{align}
	By \autoref{theorem:main-thm-2}, the representation
	\eqref{equation:moduli-space-monodromy-rep} has image with Zariski-closure $\on{SO},
	\on{Sp},$ or $\on{SL}$, depending on the self-duality properties of
	$\rho$. The connected monodromy group of \eqref{equation:kp-monodromy-rep} is in fact a \emph{normal} subgroup of this group, as the fundamental group of the fiber of the map $\mathscr{M}_t\to \mathscr{M}_{g,n-1}$ is normal in the fundamental group of $\mathscr{M}_t$. Thus if the image in question is infinite, it must be Zariski-dense in the Zariski-closure of the image of \eqref{equation:moduli-space-monodromy-rep}, as this group is simple by \autoref{theorem:simplicity-of-VHS}.
	
	We prove the infinitude of the image of this representation by
	Hodge-theoretic methods. As the monodromy representation in question is
	a topological invariant, we may take $Z$ to dominate a component of a
	general fiber of the map $\mathscr{M}_t\to \mathscr{M}_{g,n-1}$ and
	hence assume that $C$ is general.
	
	Fix general $z\in Z$, and let $X=\pi^{-1}(z)$, let $D\subset C$ be the branch locus of the cover $X\to C$, and let $p\in C$ be the image of $z$ in $C$; note that $(C,D)$ is general, and $p\in D$ by the definition of a Kodaira-Parshin fibration. 
Recall that the theorem of the fixed part implies a VHS with finite monodromy
has constant period map. Using this,
it suffices to show that the derivative of the period map $$T_zZ\to \on{Hom}(H^0(C, \widehat E^\rho_0\otimes \omega_X(D)), H^1(C, E^\rho_0))$$ is non-zero for generic $z\in Z$. There is a natural identification $$T_zZ=\ker(H^1(C, T_C(-D))\to H^1(C, T_C(-D+p))),$$ as the map $Z\to C$ is generically \'etale, so Serre-dually we wish to show that the pairing 
\begin{align*}
	&H^0(C, \widehat E^\rho_0\otimes \omega_X(D))\otimes H^0(C,
	(E^\rho_0)^\vee\otimes \omega_C) \xrightarrow{} H^0(C,
	\omega_C^{\otimes 2}(D-p))
\\ &\to \on{coker}\left(H^0(C,
	\omega_C^{\otimes 2}(D-p))
\to H^0(C, \omega_C^{\otimes 2}(D))\right),
\end{align*}
is non-zero, where the first map is induced by $B^\rho_z$, as in
	 \autoref{proposition:derivative-of-period-map}. Evaluation
	 at $p$ identifies $$\on{coker}(H^0(C, \omega_C^{\otimes 2}(D-p))\to
	 H^0(C, \omega_C^{\otimes 2}(D)))$$ with $\omega_C^{\otimes 2}(D)|_p$.
	 Using that $(C,D)$ is general, so $\widehat E^\rho_0\otimes
	 \omega_X(D)$ and $(E^\rho_0)^\vee\otimes \omega_C$ are globally generated by \autoref{proposition:globally-generated}, it is enough to show that the pairing 
	 $$(\widehat E^\rho_0\otimes \omega_X(D))\otimes (E^\rho_0)^\vee\otimes \omega_C\to \omega_X^{\otimes 2}(D)|_p$$ is non-zero. But this pairing is induced by the corresponding pairing of sheaves between $E^\rho_0\otimes \omega_X(D)$ and $(E^\rho_0)^\vee\otimes \omega_C$, which is perfect, via the inclusion $\widehat E^\rho_0\hookrightarrow E^\rho_0$. So it is non-zero as long as this inclusion is non-zero at $p$, i.e.~as long as $\rho([h])$ is non-trivial, as desired.  
\end{proof}

To complete the proof of \autoref{theorem:kodaira-monondromy}, we will need to
compute the total monodromy group, which is contained in a product of groups indexed
by the different irreps of $H$. To prove our result we will use Goursat's lemma,
and the following lemma is a key input to verifying the hypotheses of
Goursat's lemma.
\begin{lemma}\label{lemma:kodaira-separating reps}
Let $\rho_1, \rho_2: H\to \on{GL}_r(\mathbb{C})$ be irreducible $H$-representations with $\rho_i([h])$ non-trivial for $i=1,2$. If $g> r^2$ and the monodromy representations 
\begin{align}
	\label{equation:pair-monodromy-reps}
	\pi_1(Z,z)\to \on{GL}(W_1H^1(\Sigma_{g,n}, \mathbb V^{\rho_i}))
\end{align}
for $i=1,2$ are isomorphic to one another, then $\rho_1\simeq\rho_2$.
\end{lemma}
\begin{proof}
	We can factor \eqref{equation:pair-monodromy-reps} as a composition
\begin{align}
	\label{equation:pair-monodromy-reps}
	\pi_1(Z,z)\to \pi_1(\mathscr M_t) \xrightarrow{\phi_i}
	\on{GL}(W_1H^1(\Sigma_{g,n}, \mathbb V^{\rho_i}))
\end{align}
and by \autoref{proposition:functoriality}, it suffices to show that $\phi_1$ is
isomorphic to $\phi_2$,
	possibly after passing to a cover of $\mathscr{M}_t$. By \cite[Lemma
	2.2.2]{landesmanL:canonical-representations} (using that our two
	representations of $\pi_1(Z, z)$ in question are irreducible by
\autoref{lemma:kodaira-individual-reps}), the projectivizations of $\phi_1$ and
$\phi_2$
are isomorphic. Hence, we may assume $\phi_1 \simeq \phi_2 \otimes \chi$, for
$\chi$ a
character. Observe that $\chi$ must be of finite order because its
connected monodromy group is simple 
(as also argued in the penultimate paragraph of the proof of
\autoref{theorem:main-thm-1}).
Hence, we may trivialize $\chi$ by passing to a cover of $\mathscr{M}_t$, and
therefore reduce to the case $\phi_1 \simeq \phi_2$, as desired.
\end{proof}

\begin{proof}[Proof of \autoref{theorem:kodaira-monondromy}]
	The proof follows from the Goursat-Kolchin-Ribet criterion of Katz
	\cite[Proposition 1.8.2]{katz-esde}, precisely following the argument in
	the proof of \autoref{theorem:main-thm-1}; we use \autoref{lemma:kodaira-individual-reps} in place of \autoref{theorem:main-thm-2} and \autoref{lemma:kodaira-separating reps} in place of \autoref{proposition:functoriality}.
	The final statement describing \eqref{equation:kodaira} follows from
	\autoref{lemma:symplectic-centralizer}.
\end{proof}

\section{Questions}
\label{section:questions}
We conclude with a number of open questions motivated by the preceding results.

\subsection{Improving the bounds in our results}
\autoref{theorem:main-thm-1} gives a large monodromy result for the mapping class group action on the homology of an $H$-cover of a genus $g$ curve, on the assumption that $g$ is large compared to the dimensions of the irreducible representations of $H$. However, we don't know a counterexample to the statement of \autoref{theorem:main-thm-1} as long as $g\geq 3$.
\begin{question}\label{question:big-image}
	Let $H$ be a finite group and let $\Sigma_{g',n'}\to \Sigma_{g,n}$ be an $H$-cover. Suppose that $g\geq 3$. Is the identity component of the Zariski closure of the virtual image $\on{Mod}_{g,n+1}$ in $\on{GL}(W_1H^1(\Sigma_{g'}))$ the commutator subgroup of $\on{Sp}(H^1(\Sigma_{g'}, \mathbb{C}))^H$?
\end{question}

It seems natural to conjecture that the answer is yes, strengthening the
Putman-Wieland conjecture of \cite{putmanW:abelian-quotients} as explained in
\autoref{remark:putman-wieland}.
Moreover one might expect that as soon as $g\geq 3$, the image is arithmetic:
\begin{question}\label{question:arithmetic-image}
	Is the virtual image of $\on{Mod}_{g,n+1}$ in $\on{Sp}(H^1(\Sigma_{g'}, \mathbb{Z}))^H$ finite index?
\end{question}
While our methods say nothing about \autoref{question:arithmetic-image}, strengthening our results on (generic) global generation would imply a positive answer to \autoref{question:big-image} in some cases, as we now explain.
\begin{question}\label{question:global-generation}
Let $g\geq 3$, and let $\rho: \pi_1(\Sigma_{g,n})\to \on{GL}_r(\mathbb{C})$ be a representation with finite image. Fix a very general complex structure $(C,D)$ on a pointed surface $(\Sigma_g, x_1, \cdots, x_n)$ and let $E_\star$ be the parabolic bundle corresponding to $\rho$. Is the vector bundle $\widehat E_0\otimes\omega_C(D)$
\begin{enumerate}
	\item[(a)] generically globally generated, or even
	\item[(b)] globally generated?
\end{enumerate}
\end{question}

\begin{remark}
	\label{remark:gg-positive-answer}
	When $n=0$, a positive answer to \autoref{question:global-generation}(b) for
arbitrary $r$ would imply a positive answer to \autoref{question:big-image} by the methods of this paper. Similarly, for arbitrary $n$, a positive answer to \autoref{question:global-generation}(a) would imply a positive answer to the Putman-Wieland conjecture \cite{putmanW:abelian-quotients}.
\end{remark}

\begin{remark}
	\label{remark:}
	One might naturally pose \autoref{question:global-generation} for arbitrary unitary representations $\rho$. We are grateful to Eric Larson 
and Isabel Vogt for explaining to us that in this case one must necessarily take
the meaning of ``very general'' in the question to \emph{depend} on $\rho$. That
is, every smooth proper curve $C$ of genus $g\geq 2$ admits a stable vector
bundle $E$ of degree zero (corresponding, by the Narasimhan-Seshadri
correspondence, to some unitary representation) such that $E\otimes \omega_C$ is
not generically globally generated. 


\end{remark}
\begin{remark}
	\label{remark:}
	A positive answer to \autoref{question:global-generation} for general unitary representations would provide some evidence for the well-known conjecture that mapping class groups have Kazhdan's Property T in genus $g\geq 3$, as it would imply that certain unitary representations of mapping class groups are rigid, following the methods of \cite{landesmanL:canonical-representations}. See \cite[\S8]{ivanov2006fifteen} for a discussion of this question.

\end{remark}

\subsection{Arithmetic statistics}
As mentioned in \autoref{remark:arithmetic-statistics}, the results of
this paper are closely related to a number of results in arithmetic statistics,
which concern understanding the monodromy with $\mathbb Z/\ell \mathbb Z$
coefficients, instead of complex coefficients.
As noted in the introduction, \cite{jain:big-mod-ell-monodromy} is able to prove
a big monodromy result over genus $0$ bases with the number $n$ of punctures
having monodromy in every conjugacy class sufficiently large. However, his
result does not say how large $n$ has to be.
\begin{question}
	\label{question:}
	Is it possible to obtain a big monodromy result analogous to
	\autoref{theorem:main-thm-1} which is effective in $g$, or
	a result analogous to \autoref{conjecture:large-n}
	which is effective in $n$,
	with $\mathbb Z/\ell
	\mathbb Z$ coefficients instead of complex coefficients?
\end{question}
If the above were possible, can one use it to deduce some large $q$ limit versions of the Cohen Lenstra
heuristics, as alluded to in \autoref{remark:arithmetic-statistics}?

\subsection{Analogs for free groups}
Finally, it is natural to pose analogues of the questions here for representations of groups other than the mapping class group $\on{Mod}_{g,n}$. For example, the group $\on{Aut}(F_n)$ acts virtually on finite index subgroups of the free group $F_n$, and hence on their abelianizations.
\begin{question}
	Let $n\geq 3$ and fix a finite index subgroup $K\subset F_n$. Is the image in $\on{GL}(K^{\text{ab}})$ of the stabilizer of $K$ inside $\on{Aut}(F_n)$ arithmetic? What is its Zariski-closure?
\end{question}
This last question is studied in many interesting special cases by Grunewald-Lubotzky \cite{grunewald2009linear}.

\bibliographystyle{alpha}
\bibliography{bibliography-big-monodromy}

\end{document}